\documentclass[11pt]{amsart}
\usepackage{fullpage}
\usepackage{amsmath, amsthm, amssymb, graphicx, dsfont, stmaryrd, extarrows, enumitem}
\usepackage[T1]{fontenc}
\usepackage{textcomp}
\usepackage{lmodern}
\usepackage{microtype} 
\usepackage{amscd}
\usepackage{mathrsfs}
\usepackage{tikz-cd}
\usetikzlibrary{matrix,patterns}
\usepackage[colorinlistoftodos]{todonotes}
\usepackage{color}
\usepackage{bbm}
\usepackage{verbatim}
\usepackage[mathscr]{euscript}

\numberwithin{equation}{section}
\theoremstyle{plain}
\newtheorem{thm}[equation]{Theorem}
\newtheorem*{thm*}{Theorem}

\newtheorem{prop}[equation]{Proposition}
    
\newtheorem{cor}[equation]{Corollary}       
\newtheorem{lem}[equation]{Lemma}
\newtheorem{que}[equation]{Question}
\theoremstyle{definition} 
\newtheorem{defn}[equation]{Definition} 

\newtheorem{ex}[equation]{Example}

\newtheorem{rem}[equation]{Remark}

\newcommand{\Z}{\mathbb{Z}}

\newcommand{\Hom}{\mathrm{Hom}}
\newcommand{\Map}{\mathrm{Map}}

\newcommand{\mrm}[1]{\mathrm{#1}}
\newcommand{\mbb}[1]{\mathbb{#1}}
\renewcommand{\a}{\mathfrak{a}}
\newcommand{\p}{\mathfrak{p}}
\newcommand{\q}{\mathfrak{q}}
\newcommand{\m}{\mathfrak{m}}
\newcommand{\n}{\mathfrak{n}}

\newcommand{\I}{\mathcal{I}}

\newcommand{\Spech}{\mathrm{Spec}^{\mathrm{h}}}
\newcommand{\THH}{\mathrm{THH}}

\newcommand{\Ad}{\mathrm{Ad}}
\renewcommand{\L}{\mathcal{L}}
\newcommand{\res}{\mathrm{res}}
\newcommand{\Sp}{\mathrm{Sp}}
\renewcommand{\mod}[1]{\mathrm{Mod}_{#1}}

\newcommand{\Fun}{\mathrm{Fun}}
\newcommand{\pb}{\p^{\downarrow}}
\newcommand{\C}{\mathscr{C}}
\newcommand{\D}{\mathscr{D}}
\newcommand{\G}{\mathscr{G}}

\newcommand{\doublearrows}{\begin{smallmatrix} \to \\ \to \end{smallmatrix}}
\newcommand{\triplearrows}{\begin{smallmatrix} \to \\ \to \\ 
\to \end{smallmatrix} }
\newcommand{\leftdoublearrows}{\begin{smallmatrix} \leftarrow \\ \leftarrow \end{smallmatrix}}
\newcommand{\lefttriplearrows}{\begin{smallmatrix} \leftarrow \\ \leftarrow \\ 
\leftarrow \end{smallmatrix} }
\usepackage{hyperref}
\hypersetup{
	colorlinks=true,
	citecolor=blue,
	linkcolor=blue,
	filecolor=magenta,      
	urlcolor=cyan,
}

\makeatletter
\@namedef{subjclassname@2020}{\textup{2020} Mathematics Subject Classification}
\makeatother
\begin{document}
\title{Local Gorenstein duality in chromatic group cohomology}
\author{Luca Pol}
\address[Pol]{Fakult\"{a}t f\"{u}r Mathematik, Universit\"{a}t Regensburg, Universit\"{a}tsstra{\ss}e 31, 93053 Regensburg, Germany}
\email{luca.pol@ur.de}
\author{Jordan Williamson}
\address[Williamson]{Department of Algebra, Faculty of Mathematics and Physics, Charles University in Prague, Sokolovsk\'{a} 83, 186 75 Praha, Czech Republic}
\email{williamson@karlin.mff.cuni.cz}

\begin{abstract}
We consider local Gorenstein duality for cochain spectra $C^*(BG;R)$ on the classifying spaces of compact Lie groups $G$ over complex orientable ring spectra $R$. We show that it holds systematically for a large array of examples of ring spectra $R$, including Lubin-Tate theories, topological $K$-theory, 
and various forms of topological modular forms. We also  
prove a descent result for local Gorenstein duality which allows us to access further examples.
\end{abstract}
\maketitle
\setcounter{tocdepth}{1}
\tableofcontents

\section{Introduction}
Many algebraic constructions and ideas can be understood in a homotopy invariant sense, and as such can be considered in the realm of spectra. 
This passage of techniques and ideas between spectral algebra and homological algebra, has led to a powerful new point of view in recent years, 
in particular with regards to modular representation theory~\cite{BGduality, BIK3, DGI, Greenlees2018}. 

The importance of the Gorenstein condition in commutative algebra led Dwyer-Greenlees-Iyengar~\cite{DGI} to introduce the notion of a Gorenstein ring spectrum. Ordinary Gorenstein rings satisfy a duality property in terms of their local cohomology and this in fact characterizes Gorenstein rings. This Gorenstein duality property is suitably homotopy invariant, and therefore can be understood for ring spectra, but in topology, the relation between the Gorenstein condition and the duality property is less strict. Dwyer-Greenlees-Iyengar~\cite{DGI} showed that Gorenstein ring spectra often do satisfy an analogous duality property, but that this is not automatic. Therefore, it can be useful to isolate this duality phenomenon, and seek ring spectra which satisfy \emph{local Gorenstein duality}~\cite{BHV, BGduality}. 

In this paper we will study the local Gorenstein duality property for cochain spectra on classifying spaces of compact Lie groups with coefficients in general
commutative noetherian ring spectra, in particular, for complex orientable coefficients.

\subsection*{Local Gorenstein duality} 
Recall that a commutative noetherian local ring $(R, \m,k)$ is said to be \emph{Gorenstein} if $R$ has finite injective dimension as an $R$-module. A local ring is Gorenstein if and only if it satisfies \emph{Gorenstein duality}; that is, the local cohomology of $R$ satisfies
\[H_\m^i(R) \cong \begin{cases} I_\m & i=\mrm{dim}R \\ 0 & \mrm{else}\end{cases}\] where $I_\m$ is the injective hull of the residue field $k$. 
As localization is an exact functor which preserves injectives (for noetherian rings), one sees that if $R$ is Gorenstein, then $R_\p$ is also Gorenstein for all $\p \in \mrm{Spec}(R)$. Therefore, if $R$ is Gorenstein, for any $\p \in \mrm{Spec}(R)$ one obtains a \emph{local Gorenstein duality} statement
\[H_\p^i(R_\p) \cong \begin{cases} I_\p & i=\mrm{dim}(R_\p) \\ 0 & \mrm{else}\end{cases}\] where $I_\p$ is the injective hull of $R/\p$.

Given a commutative noetherian ring spectrum $R$ and a homogeneous prime ideal $\p$ of $\pi_*R$, one can construct functors $\Gamma_\p$ and $(-)_\p$ which are spectral analogues 
of local cohomology and localization~\cite{GMcompletions, BIK}. 
Then one says that the ring spectrum $R$ satisfies \emph{local Gorenstein duality}~\cite{BHV, BGduality} if for each $\p \in \Spech(\pi_*R)$ there exists an integer $\nu(\p)$ such that
\[\Gamma_\p R_\p \simeq \Sigma^{\nu(\p)}\I_\p\]
where $\I_\p$ is the Brown-Comenetz spectrum of the injective hull $I_\p$. 
The integers $\nu(\p)$ can be organized into a shift function $\nu\colon\Spech(\pi_*R)\to \Z$ on the spectrum of prime ideals.

We emphasize that if $\pi_*R$ is Gorenstein, then $R$ has local Gorenstein duality, but the converse is far from true, 
so that this definition does capture more information than a naive definition on homotopy groups. The Gorenstein duality property yields a strongly convergent spectral sequence
\[H_\p^{-s}((\pi_*R)_\p)_t \Rightarrow (I_\p)_{s+t-\nu(\p)}\] which has striking implications for the coefficient ring 
$\pi_*R$ as investigated by Greenlees-Lyubeznik~\cite{GL}. For example, $\pi_*R$ is Cohen-Macaulay if and only if it is Gorenstein, 
and $\pi_*R$ is generically Gorenstein.

\subsection*{Cochain spectra}
One key example where these duality properties starkly manifest themselves is for 
cochain spectra $C^*(BG;R):=\Hom(BG_+, R)$ where $G$ is a compact Lie group and $R$ is a commutative noetherian ring spectrum. This is a commutative ring spectrum, so one may proceed to apply the methods of commutative algebra to $C^*(BG;R)$. The case where $R=k$ is a field and $G$ is a finite group has been studied in depth by many authors~\cite{DGI, BGduality, Benson, BHV}. In this case, the cochain spectrum $C^*(BG;k)$ was shown to have local Gorenstein duality, and this duality recovers Benson-Carlson duality~\cite{BensonCarlson}. Moreover, Benson-Greenlees~\cite{BGduality} showed that once this result is translated to the stable module category, this confirms a conjecture of Benson~\cite{Bensonconj}. These results also hold more generally for compact Lie groups with an additional orientability hypothesis. 

The goal of this paper is to vastly generalize these duality results by considering 
cochain spectra over complex orientable ring spectra $R$. We provide conditions under which $C^*(BG;R)$ satisfies local Gorenstein duality
 and then give various examples coming from chromatic homotopy theory in which our main theorem applies. 
 Therefore, one can view our main result as saying that a chromatic analogue of Benson's conjecture also holds. 
For convenience, we only state our main result for finite groups in this introduction and refer the reader to Theorem~\ref{thm:duality} in the body of the paper for the case of more general compact Lie groups.

\begin{thm*}
Let $G$ be a finite group and $R$ be a commutative ring spectrum with $\pi_*R$ even, noetherian, and of finite global dimension. Then the cochain spectrum $C^*(BG;R)$ has local Gorenstein duality and the shift function on maximal ideals $\m$ of $R^*(BG)$ takes the form
\[
\m \mapsto \nu(\m \cap \pi_*R)-  \dim(\pi_*R_{\m \cap \pi_*R})
\]
where $\nu$ denotes the Gorenstein shift function of the graded ring $\pi_*R$.
\end{thm*}

A large array of ring spectra to which this theorem applies is given in Section~\ref{sec:examples}. In particular, this includes complex $K$-theories, Lubin-Tate spectra and various forms of topological modular forms.

We now provide a rough outline of how the proof proceeds. 

\subsection*{The main ingredients} 
Let us consider a commutative ring spectrum $R$ with $\pi_*R$ even, noetherian and of finite global dimension. Consider a finite group $G$; we also treat the case when $G$ is a compact Lie group, but we restrict to the finite case in this introduction for simplicity.

We may embed $G$ into a unitary group $U$, and consider the induced ring map 
\[
f\colon C^*(BU;R) \to C^*(BG;R).
\]
Since $\pi_*R$ is even, we may choose a complex orientation for $R$ and describe 
the homotopy groups of $C^*(BU;R)$ as a power series ring over the ring $\pi_*R$. As these homotopy groups are regular and hence Gorenstein, we can immediately deduce local Gorenstein duality for the cochain spectrum $C^*(BU;R)$. On the other hand, the homotopy of $C^*(BG;R)$ is often difficult to calculate and in general will not be regular (nor even Gorenstein).
We can overcome this problem by applying the ascent theorem for local Gorenstein duality of Barthel-Heard-Valenzuela~\cite{BHV} to the map $f$ above.

There are two main conditions that one needs to verify. 

Firstly, we show that the map $f$ is finite, or in other words, that $C^*(BG;R)$ is a compact $C^*(BU;R)$-module. We prove this by combining the finite global dimension assumption on $\pi_*R$ with a result of Venkov, which is classically only stated for $R$ a discrete ring but that holds more generally for any complex orientable ring spectrum, see Theorem~\ref{thm:venkov2}. 

Secondly, we show that the map $f$ is \emph{relatively Gorenstein} of shift $d=\mrm{dim}(U/G)$, meaning that
\[
\Hom_{C^*(BU;R)}(C^*(BG;R), C^*(BU;R)) \simeq \Sigma^d C^*(BG;R).
\]
In fact, one may also deduce that the relative Gorenstein condition holds more generally for the map $C^*(BG;R) \to C^*(BH;R)$ where $H$ is a subgroup of $G$, see Theorem~\ref{thm:relGor2}. The assumption that $\pi_*R$ is even (as opposed to $R$ being complex orientable) is used in this step to ensure that a certain orientability condition for the adjoint representation of $U$ is satisfied. This can be verified by showing that certain differentials in the homotopy fixed points spectral sequence are zero, see Lemma~\ref{lem:LUorientable}. The proof of the relatively Gorenstein condition relies on equivariant techniques and uses a certain generation result for modules over Borel spectra. More precisely, we show that for a complex orientable ring spectrum $R$, the genuine $U$-equivariant Borel spectrum $b_UR$ generates its category of modules, see Theorem~\ref{thm:unitary}. We note that a priori the module category is generated by the module spectra $U/H_+ \wedge b_UR$ as $H$ ranges over the closed subgroups of $U$. It is therefore surprising that in this case only $b_U R$ is required. 
We prove this result by combining a method of Greenlees~\cite{Greenlees2020} used in the case when $R$ is an ordinary commutative ring, with the unipotence results of Mathew-Naumann-Noel~\cite{MathewNaumannNoel17}. 

\subsection*{Descent}
We also prove a descent theorem for local Gorenstein duality which allows us to deduce further duality phenomena. The motivating example is the complexification map $f\colon KO \to KU$ between $K$-theory spectra. Our descent result applies more generally to maps which are finite and descendable, see Definition~\ref{def-descendable}. As discussed in~\cite{MathewGalois}, these maps can be viewed as a generalization of faithful Galois extensions.

In order to descend local Gorenstein duality along a finite descendable map $f\colon R \to S$ of commutative noetherian ring spectra we will require the shift function of $S$ to be of a certain form as we now describe. Given $\p \in \Spech(\pi_*R)$, we let $\pb$ denote a prime ideal of $\pi_*S$ such that $f^{-1}(\pb) = \p$. Since $f$ is finite the lying over theorem ensures that such a prime does exist, but in general there will be many choices of such a prime. We therefore require the shift function of $S$ to be locally constant in the sense that each such choice of $\pb$ has the same shift. The following descent statement appears in the body of the paper as Theorem~\ref{thm:descent}. 

\begin{thm*}
Let $f\colon R \to S$ be a map of commutative noetherian ring spectra such that: 
\begin{itemize}
\item[(a)] $f\colon R \to S$ is relatively Gorenstein of shift $a\in \Z$;
\item[(b)] $f$ is finite and descendable;
\item[(c)] $S$ has local Gorenstein duality with shift function $\nu$ locally constant relative to $f$.
\end{itemize}
Then $R$ has local Gorenstein duality with shift function $\p \mapsto \nu(\pb) + a$.
\end{thm*}

In particular we apply this result to the complexification maps $KO \to KU$ and $ko \to ku$ and deduce local Gorenstein duality for $KO$ and $ko$. We note that $ko$ was already known to have Gorenstein duality at maximal 
  ideals by~\cite[18.1, 19.1]{Greenlees2018}. The fact that this holds for 
  all prime ideals as well as for the periodic version $KO$ seems to be new.   

\subsection*{Relation to other work} 
The idea of investigating Gorenstein duality phenomena in ring spectra appears in several places, see~\cite{GreK, BG97} for instance. Dwyer-Greenlees-Iyengar~\cite{DGI} gave the first systematic study of this, leading to a powerful new point of view on such duality statements. 
The Gorenstein duality property in modular representation theory at maximal ideals was shown in~\cite{Gre95, BG97, DGI}, and Benson-Greenlees~\cite{BGduality} then showed that this duality holds at all prime ideals by
extending algebraic work of Greenlees-Lyubeznik~\cite{GL} on dual localization to the spectral world. Barthel-Heard-Valenzuela~\cite{BHV} 
made a general definition of local Gorenstein duality for arbitrary commutative noetherian ring spectra and established its key properties. Together with Castellana, they also studied local Gorenstein duality for cochains on $p$-compact groups~\cite{BCHV2} and for more general spaces~\cite{BCHV} all with field coefficients. The equivariant approach of studying Gorenstein duality for cochains was employed by Benson-Greenlees~\cite{BensonGreenlees14} although there were some inaccuracies there which were later corrected by Greenlees~\cite{Greenlees2020}. Our approach owes a debt to all of the above-mentioned works.

\subsection*{Conventions} 
Let $\C$ be a presentable, symmetric monoidal stable $\infty$-category for which the tensor product 
$\otimes$ commutes with colimits separately in each variable. 
An object $X \in \C$ is \emph{compact} if the natural map of mapping spectra
\[\bigoplus \Hom_\C(X, Y_i) \to \Hom_\C(X, \bigoplus Y_i)\]
is an equivalence for any collection $\{Y_i\}$ of objects of $\C$. A full, replete subcategory $\D$ of $\C$ is \emph{thick} if it is closed under suspensions, triangles and retracts, and \emph{localizing} if it is thick and closed under arbitrary sums. We say that $\D$ is an \emph{ideal} if it is closed under tensoring with arbitrary objects of $\C$. Given a set of objects $\G \subset \C$ we write $\mrm{Thick}(\G)$ (resp., $\mrm{Loc}(\G)$, resp., $\mrm{Thick}^\otimes(\G)$, resp., $\mrm{Loc}^\otimes(\G)$) for the smallest thick subcategory (resp., localizing subcategory, resp., thick ideal, resp., localizing ideal) of $\D$ containing $\G$. For objects $X,Y \in \C$, if $Y \in \mrm{Thick}(X)$ we say that $X$ \emph{finitely builds} $Y$, and if $Y \in \mrm{Loc}(X)$ we say that $X$ \emph{builds} $Y$. A set $\G$ of objects of $\C$ is said to \emph{generate} $\C$ if $\mrm{Loc}(\G) = \C$.

We work in the underlying symmetric monoidal $\infty$-category of orthogonal $G$-spectra (and non-equivariant spectra) as discussed in~\cite[\S 5]{MathewNaumannNoel17}. A commutative ring spectrum is the same as an $\mbb{E}_\infty$-ring spectrum. It is noetherian if its homotopy groups form a graded noetherian ring. By a complex oriented commutative ring spectrum we always mean a commutative ring spectrum with an $\mbb{E}_1$-complex orientation. For a commutative ring spectrum $R$, we will write $\Hom_R(-,-)$ for the internal hom in $R$-modules. We will use the same notation even if $R$ is a commutative ring $G$-spectrum. The internal hom in spectra and $G$-spectra will be denoted by $\Hom(-,-)$, and we write $D$ for the functional dual $D = \Hom(-, S^0)$. Subgroups of compact Lie groups are always assumed to be closed.

Our grading conventions are homological so that differentials decrease degree by one. Degrees (subscripts) are converted to codegrees (superscripts) via negation. The suspension functor is cohomological so that $\pi_n(\Sigma^i X) = \pi_{n-i} X$.

\subsection*{Acknowledgements}
We are grateful to  Tobias Barthel, John Greenlees, Niko Naumann and Liran Shaul for several helpful suggestions and discussions. The first author thanks the SFB 1085 Higher Invariants in Regensburg for support.
The second named author was supported by the grant GA~\v{C}R 20-02760Y from the Czech Science Foundation.

\section{Local Gorenstein duality in algebra}
In this section we discuss some important algebraic preliminaries which we will require in order
to discuss local Gorenstein duality for ring spectra. We assume that the reader is familiar with local cohomology (in the classical algebraic setting) and refer the reader to~\cite[Appendix A]{Greenlees2018} for a brief overview of the definition and key properties.

\subsection{Equicodimensionality and the dimension formula}

Throughout we will work with a graded commutative noetherian ring $R$. In concrete terms this means that $R_0$ is noetherian and that $R$ is a finitely generated $R_0$-algebra. 
Let us start by recalling the following terminology. 

\begin{defn}
Let $R$ be a (graded) commutative ring and $\p$ a (homogeneous) prime ideal of $R$.
\begin{itemize}
\item The height of $\p$ is the dimension of the ring $R_\p$. We will denote the height of $\p$ by $\mrm{ht}_R(\p)$ or simply $\mrm{ht}(\p)$ when the ring is clear 
from the context.
\item The coheight of $\p$ is the dimension of the quotient ring $R/\p$. We will 
denote the coheight of $\p$ by $\mrm{coht}_R(\p)$ or $\mrm{coht}(\p)$.
\item The ring $R$ is said to be \emph{equicodimensional} if each of its maximal ideals has the same height.
\end{itemize}
\end{defn}

\begin{ex}
Any local ring is equicodimensional. The integers are also equicodimensional. 
\end{ex}

Recall that a local noetherian ring is said to be \emph{Cohen-Macaulay} if its depth is equal to its dimension, or equivalently if its local cohomology is concentrated in a single degree. A non-local ring is said to be Cohen-Macaulay if each localization at a prime ideal is Cohen-Macaulay.

\begin{lem}[{\cite[2.1.4]{BrunsHerzog}}]\label{lem:equi}
If $R$ is a local Cohen-Macaulay ring, then for each prime ideal $\p$ of $R$, there is an equality \[\mrm{ht}(\p) + \mrm{coht}(\p) = \mrm{dim}(R).\]
\end{lem}

\begin{rem}\label{rem:error}
We emphasize that this equality need not hold if $R$ is not local (even if it is Cohen-Macaulay, or even regular). For example let $(R,\m)$ be a discrete valuation ring whose maximal ideal has uniformizing parameter $\pi$, i.e., $\m=(\pi)$. The ring $R[T]$ has dimension $2$ and for the maximal ideal $\p =(\pi T-1)$, the dimension formula is false: 
\[
\mrm{ht}(\p)+\mrm{coht}(\p)=1+0=1 \not =2=\mrm{dim}(R[T]).
\]
We thank the StackExchange user Georges Elencwajg for providing this counterexample~\cite{stack}.
The failure of this equality when $R$ is not local is behind the error in the proof of~\cite[4.6]{BHV}.
\end{rem}

We are interested in the equicodimensional property because of its relation to the dimension formula.
\begin{prop}\label{prop:dimformula}
Let $R$ be a Cohen-Macaulay ring. Then $R$ is equicodimensional if and only if $\mrm{ht}(\p) + \mrm{coht}(\p) = \mrm{dim}(R)$ for all prime ideals $\p$ of $R$.
\end{prop}
\begin{proof}
The backward implication follows immediately from the fact that the coheight of a maximal ideal is zero. For the forward implication, fix a prime ideal $\p$ of $R$ and a maximal ideal $\m\supset \p$. By Lemma~\ref{lem:equi}, we have $\mrm{ht}(\p) + \mrm{dim}(R_\m/\p) = \mrm{dim}(R_\m)$. Since $R$ is equicodimensional we have $\mrm{dim}(R) = \mrm{dim}(R_\m)$ and so $\mrm{dim}(R_\m/\p) = \mrm{dim}(R)-\mrm{ht}(\p)$. We always have $\mrm{ht}(\p)+ \mrm{dim}(R/\p)\leq \mrm{dim}(R)$ and so 
we deduce that 
\[
\mrm{dim}(R/\p)\leq \mrm{dim}R-\mrm{ht}(\p)=\mrm{dim}(R_\m/\p).
\]
Clearly, $\mrm{dim}(R_\m/\p)\leq \mrm{dim}(R/\p)$ so we conclude that 
$\mrm{dim}(R_\m/\p)= \mrm{dim}(R/\p)$.  Therefore the dimension formula holds for $R$. 
\end{proof}

We now turn to considering the behaviour of coheight along finite ring maps. Given a map $f\colon R \to S$ of commutative noetherian rings, we write $\mrm{res}_f\colon \Spech(S) \to \Spech(R)$ for the induced map on prime spectra.
\begin{lem}\label{lem:coheight}
  Let $f\colon R \to S$ be a map of commutative noetherian rings such that $S$ is 
  a finitely generated $R$-module. Let $\q$ be a prime ideal of $S$ and write 
  $\p = \res_f(\q)$ for the induced prime ideal of $R$. 
  Then $\mrm{coht}_R(\p) = \mrm{coht}_S(\q)$.
\end{lem}

\begin{proof}
 Consider a chain of prime ideals 
 $\p \subset \p_1 \subset \p_2 \subset \cdots \subset \p_n$ of $R$. 
 Since $f$ is finite, the going-up theorem holds~\cite[4.9]{Johnson} and so there exists
 a chain of prime ideals 
 $\q \subset \q_1 \subset \q_2 \subset \cdots \subset \q_n $ of $S$, giving 
 $\mrm{coht}(\p) \leq \mrm{coht}(\q)$. 
 
 Conversely, if $\q \subseteq \q'$ are prime ideals of $S$ such that 
  $\res_f(\q) = \res_f(\q')$, then $\q = \q'$ by the incomparability theorem~\cite[4.7]{Johnson}.
  This shows that the inverse image of a chain of primes in $S$ containing $\q$ is a 
  chain of primes in $R$ containing $\p$ in which the inclusions are strict. 
  Therefore $\mrm{coht}(\p) \geq \mrm{coht}(\q)$ as required.
\end{proof}

\subsection{Graded Gorenstein rings}
We now discuss some aspects of graded Gorenstein rings which are relevant to the
notion of local Gorenstein duality of ring spectra. Our approach closely follows~\cite{BHV, GL}, but 
we have to give modified versions of some definitions and results in order to treat the examples which arise
later on in this paper.

Given a prime ideal $\p$ of $R$, we write $I_\p$ for the injective hull of $R/\p$. 
\begin{rem}
We note that $I_\p$ is naturally an $R_\p$-module, and that there is an isomorphism of $R_\p$-modules between $I_\p$ and the injective hull of $R_\p/\p R_\p$,
see~\cite[3.77]{Lam}. 
\end{rem}

\begin{defn}\label{def-graded-gorenstein}\leavevmode
\begin{itemize}
\item A graded commutative noetherian local ring $(R,\m)$ is \emph{graded Gorenstein} if there exists $\nu(\m)\in \Z$ such that
\[
H^i_{\mathfrak{m}}(R)=
\begin{cases} 
\Sigma^{\nu(\m)}I_\m &\mathrm{if}\;\; i=\mrm{dim}(R) \\
0 & \mathrm{if} \;\; i \not = \mrm{dim}(R) 
\end{cases}
\]
 In this case, we say that $R$ is a \emph{graded Gorenstein ring of shift $\nu(\m)$}. 
 \item A graded commutative ring $R$ is \emph{graded Gorenstein} if $R_\m$ is a graded local Gorenstein ring for each maximal ideal $\m$ of $R$. The values $\nu(\m)$ assemble to give a function $\nu\colon \mrm{MaxSpec}^\mrm{h}(R) \to \mathbb{Z}$ which we call the \emph{shift function} of $R$.
\end{itemize}
\end{defn}

\begin{rem}\label{rem:comparison}
This definition is different from that given in~\cite[4.5]{BHV} where the shift function
$\nu$ must be of the form $\nu(\m) = \mrm{ht}(\m) + c$, for a fixed $c\in\Z$. The definition given in~\cite{BHV} works well when $R$ is local, but has two problematic consequences otherwise:
\begin{itemize}
\item If $R$ is an ungraded Gorenstein ring, then it need not be graded Gorenstein in the sense of~\cite{BHV}. 
More precisely, $R$ is graded Gorenstein in the sense of~\cite{BHV} if and only if $R$ is equicodimensional.
\item If $R$ is a graded Gorenstein ring and $\p\in \Spech(R)$, then $R_\p$ need not be graded Gorenstein in the sense of~\cite{BHV}. 
The error in~\cite[4.6]{BHV} is that the dimension formula need not hold in general, see Remark~\ref{rem:error}.
\end{itemize}
On the other hand, the definition we give above does not suffer from these issues.
\end{rem}

\begin{ex}
Let $R$ be an ungraded Gorenstein ring. Then $R$ is a graded Gorenstein ring with shift 
$\nu = 0$, the zero function. 
\end{ex}

The following result is well-known, see for example~\cite[18.2]{Matsumura1989}. 

\begin{thm}
A graded commutative ring $R$ is graded Gorenstein if and only if $R_\p$ is graded Gorenstein for each $\p \in \Spech(R)$. Therefore the shift function $\nu\colon \mrm{MaxSpec}^\mrm{h}(R) \to \Z$ of $R$ extends to a function $\Spech(R) \to \Z$ which we still denote by $\nu$, with the property that
\[H^i_\p(R_\p) = \begin{cases} 
\Sigma^{\nu(\p)}I_\p &\mathrm{if}\;\; i=\mrm{dim}(R_\p) \\
0 & \mathrm{if} \;\; i \not = \mrm{dim}(R_\p). 
\end{cases} \]
\end{thm}

The rest of this subsection will be dedicated to proving the following result.

\begin{thm}\label{thm:localize}
 If $R$ is a coconnected graded Gorenstein ring with $R_0 = k$ a field, then the shift 
 function $\nu$ is constant, i.e., $\nu(\p)=\nu(\p')$ for all $\p,\p'\in\Spech(R)$.
\end{thm}

In order to prove the previous theorem, we must firstly give some necessary background. Recall that since local cohomology modules are torsion, we cannot directly localize them, and must instead use dual localization. 
\begin{defn}
Let $(R,\m)$ be a graded local ring, $\p\in\Spech(R)$ and $I_\p$
the injective hull of $R/\p$.
\begin{itemize}
\item The \emph{Matlis dual} of an $R_\p$-module $M$ is given by 
$D_\p M = \mrm{Hom}_{R_\p}(M, I_\p)$.
\item The \emph{dual localization} of an $R$-module $M$ is given by 
$\L_\p M = \mrm{Hom}_{R_\p}(\mrm{Hom}_R(M, I_\m)_\p, I_\p)$. In other words,
$\L_\p\colon \mod{R} \to \mod{R_\p}$ is defined by the composite \[\mod{R} \xrightarrow{D_\m} \mod{R}^{\mrm{op}} \xrightarrow{(-)_\p} \mod{R_\p}^{\mrm{op}} \xrightarrow{D_\p} \mod{R_\p}.\] \end{itemize}
\end{defn}

As Matlis duality is a crucial ingredient in proving many results about local 
cohomology, we recall some key facts here and refer the reader to~\cite[\S 3]{Huneke} or \cite[3.6.17]{BrunsHerzog} for more details. 

Consider a complete local ring $(R,\m)$.
\begin{itemize}
\item The Matlis dual functor $D_\m$ gives a bijection between artinian $R$-modules and finitely generated $R$-modules. 
\item The injective hull $I_\m$ and $R$ are Matlis 
dual; that is, $D_\m I_\m = R$.
\end{itemize}
Using these facts one obtains the following proposition.
\begin{prop}[{\cite[2.5]{GL}}]\label{prop:duallocalization}
Let $R$ be a coconnected graded noetherian ring with $R_0 = k$ a field, $M$ be a finitely generated $R$-module and $\p$ be a prime ideal of $R$. Then 
\[
\L_\p I_\m = I_\p \quad \mathrm{and} \quad
\L_\p H^i_\m(M) = H^{i - \mrm{coht}(\p)}_\p (M_\p).
\]
\end{prop}  
\begin{proof}
Note that the assumptions force $R$ to be local with maximal ideal $\m$ given by the negative degree elements, and moreover that $R$ is $\m$-adically complete. The first isomorphism then follows from Matlis duality since 
\[
\L_\p I_\m=D_\p((D_\m I_\m)_\p)= D_\p(R_\p) = I_\p.
\]
The second isomorphism is~\cite[2.5]{GL}.
\end{proof}

We now have all the necessary background to prove Theorem~\ref{thm:localize}.

\begin{proof}[Proof of Theorem~\ref{thm:localize}]
Fix $R$ to be coconnected with $R_0 = k$ a field. Note that $R$ is then local with maximal ideal $\m$ given by the negative degree elements. Let $\p$ be a prime ideal of $R$. We now show that $\nu(\p) = \nu(\m)$. Since $R$ is Gorenstein we have $H_\m^{\mrm{dim}(R)}(R) = \Sigma^{\nu(\m)}I_\m$. Applying dual localization $\L_\p$ to this together with Proposition~\ref{prop:duallocalization}, we have
\[H_\p^{\mrm{dim}(R_\p)}(R_\p)=\L_\p H_\m^{\mrm{dim}(R)}(R) = \L_\p(\Sigma^{\nu(\m)}I_\m) = \Sigma^{\nu(\m)}I_\p\]
as required.
\end{proof}

\begin{rem}\label{rem-counterexample}
For a general graded local Gorenstein ring $R$, the shift function $\nu$ of $R$ need not be constant in the sense of Theorem~\ref{thm:localize}. For example, one may take $R = \Z_{(p)}[x]$ with $x$ in degree $2$. One easily calculates that $\nu((p,x)) = -2$ whereas $\nu((p)) = 0$.
\end{rem}

\subsection{Arithmetic of shifts}
In this section, we calculate the shift functions for some important types of graded Gorenstein rings. 

\begin{prop}\label{prop:Gorquotient}
Let $R$ be a graded local ring, and $x_1,\ldots,x_n$ be a regular sequence in $R$. Write $d_i$ for the degree of $x_i$. Then $R/(x_1,\ldots ,x_n)$ is Gorenstein if and only if $R$ is Gorenstein. Moreover the shifts on the maximal ideals correspond as follows: $R/(x_1,\ldots ,x_n)$ has shift $\nu$ if and only if $R$ has shift $\nu - \sum_{i=1}^n d_i$.
\end{prop}

\begin{proof} 
This can be found in~\cite[4.6]{Huneke}. The cited reference proves the claim for ungraded rings but the same argument works for graded rings; that is, for a regular element $x$ we have the short exact sequence \begin{equation}\label{eq-ses}
0 \to \Sigma^{|x|}R \xrightarrow{x} R \to R/x \to 0
\end{equation}
and the claim follows by applying the long exact sequence in local cohomology. 
\end{proof}

\begin{rem} 
 One could hope that the relation between the shift functions of $R$ and $R/(x_1,\ldots ,x_n)$ holds for all primes and not for just for maximal 
 ideals as stated in the previous proposition. This is false: take $\Z_{(p)}[x]$ with $|x|=2$. The element $x$ is regular with quotient $\Z_{(p)}$ which is Gorenstein of shift $0$. 
 However $\Z_{(p)}[x]$ does not have constant shift $-2$ as shown in 
 Remark~\ref{rem-counterexample}. We can see what goes wrong in the proof of Proposition~\ref{prop:Gorquotient} by localizing the exact sequence (\ref{eq-ses}) at a prime ideal 
 $\p$. If $x\not \in \p$ then $R_\p/x=0$ so we cannot argue by induction. 
\end{rem}
 
\begin{cor}\label{cor-poly-shift-gor}
 Let $R$ be a graded Gorenstein ring with shift function $\nu$ and consider the graded ring $R\llbracket x_1, \ldots ,x_n\rrbracket$ where each $x_i$ has even non-zero degree $d_i$ of the same parity. Then $R\llbracket x_1,\ldots ,x_n\rrbracket$ is Gorenstein with shift function $\sigma(\m) = \nu(\m \cap R) - \sum_{i=1}^n d_i$ for all $\m\in \mrm{MaxSpec}^\mrm{h}(R\llbracket x_1,\ldots ,x_n\rrbracket )$. 
 \end{cor}
 
\begin{proof}
 The fact that $R\llbracket x_1,\ldots ,x_n\rrbracket$ is Gorenstein follows easily 
 from Proposition~\ref{prop:Gorquotient}.
 Consider a maximal ideal $\m\in \Spech(R\llbracket x_1, \ldots, x_n\rrbracket )$.  
 We claim that $(x_1,\ldots, x_n) \subset \m$. 
 Suppose that this is not the case and note that $(x_1, \ldots,x_n, \m)$ is again 
 prime in $R\llbracket x_1,\ldots,x_n \rrbracket$. 
 This contradicts the maximality of $\m$ so $(x_1,\ldots,x_n) \subset \m$ and 
 $\m=(\m \cap R,x_1,\ldots,x_n)$. To calculate the shift we must consider $R\llbracket x_1,\ldots ,x_n\rrbracket_\m = R_{\m \cap R}\llbracket x_1,\ldots,x_n\rrbracket$.
 By Proposition~\ref{prop:Gorquotient} this is Gorenstein of shift 
 $\sigma(\m)=\nu(\m \cap R)-\sum_{i=1}^n d_i$ as required.  
\end{proof}
 
 Finally, we consider the case when the graded ring is periodic.
 \begin{prop}\label{prop:periodic}
Consider $R = r[\beta^{\pm 1}]$, where $r$ is an ungraded noetherian ring, and $\beta$ has non-zero even degree.
If $r$ is Gorenstein, then $R$ is Gorenstein of shift $0$.
\end{prop}
\begin{proof}
Firstly, note that $\Spech(R) = \mrm{Spec}(r)$. For all prime ideals $\p$ of $r$ we will show that $R_\p$
is Gorenstein of shift $0$. As $r[\beta^{\pm 1}]$ is a flat $r$-module, the change of base property of local cohomology gives that $H^*_{\p}(R_\p)=H^*_\p(r_\p)[\beta^{\pm 1}]$. Therefore $H^*_{\p}(R_\p)$ is concentrated in $\ast = \mrm{dim}(r_\p) = \mrm{dim}(R_\p)$. Writing $I^r_\p$ and $I^R_\p$ for the injective hulls of $r/\p$ and $R/\p$ respectively, we have
\[
H_\m^{\mrm{dim}(R_\p)}(R_\p) =  H^{\mrm{dim}(r_\p)}_\m(r_\p)[\beta^{\pm 1}] = I^r_\p[\beta^{\pm 1}] = I^R_\p.
\]
\end{proof}

\section{Local Gorenstein duality for ring spectra}

In this section we recall the definition of local Gorenstein duality for ring spectra from~\cite{BHV, BCHV}. 
We will also recall some methods of proving local Gorenstein duality, in particular, the ascent theorem of Barthel-Heard-Valenzuela.

\subsection{Local cohomology and localization}
Firstly we recall the definition of local cohomology for ring spectra together 
with some key properties. 
We refer the reader to~\cite{GMcompletions},~\cite[Appendices A and B]{Greenlees2018},~\cite{BIK} and~\cite[\S 3]{BHV} for more details.

Let $R$ be a commutative noetherian ring spectrum and consider a homogeneous ideal 
$\a = (x_1, \ldots, x_n)$ of $\pi_*R$. Mirroring the construction of the stable Koszul complex in algebra, we define the stable Koszul spectrum by 
\[
K_\infty(\a) = \bigotimes_{i=1}^n \mrm{fib}(R \to R[1/x_i])
\]
where $R[1/x_i] = \mrm{colim}(R\xrightarrow{\cdot x_i} R \xrightarrow{\cdot x_i} \cdots)$. One may show that the homotopy type of $K_\infty(\a)$ is independent of the choice of generators. In fact it depends only on the radical of $\a$. We then define $\Gamma_\a M = K_\infty(\a) \otimes_R M$ for any $R$-module $M$, and call this the $\a$-local cohomology functor. This terminology is justified by the existence of the strongly convergent spectral sequence~\cite[3.2]{GMcompletions} 
\begin{equation}\label{eq:lcss} E^2_{s,t} = H_\a^{-s}(\pi_*M)_t \Rightarrow \pi_{t-s}(\Gamma_\a M)\end{equation}
from the $\a$-local cohomology of $\pi_*M$.

In a similar way, given a prime ideal $\p$ of $\pi_*R$, one may define a localization functor $L_\p$ with the property that $\pi_*(L_\p M) = (\pi_*M)_\p$ for all $R$-modules $M$. For this reason we will often write $M_\p$ for $L_\p M$. Both $L_\p$ and $\Gamma_\p$ are smashing and hence $\Gamma_\p (M_\p) \simeq (\Gamma_\p M)_\p.$

\begin{rem}
We warn the reader that~\cite{BCHV, BHV, BIK} use $\Gamma_\p$ to denote the composite $\Gamma_\p L_\p$ which we do not do. Instead we choose to keep the notation for the local cohomology and localization functors distinct.
\end{rem}

\subsection{Local Gorenstein duality and consequences}
In this section we recall the definition of local Gorenstein duality from~\cite{BHV}.

First we recall the construction of Brown–Comenetz spectra which play the role of 
injectives in the derived world.

\begin{lem}\label{lem:matlis}
Let $R$ be a commutative noetherian ring spectrum and let $I$ be an injective $\pi_*R$-module. 
There exists an $R$-module $T_R(I)$ satisfying:
\begin{itemize}
\item[(a)] the assignment $T_R(-)$ is functorial in $I$;
\item[(b)] for all $R$-modules $M$, we have $\pi_*\Hom_R(M, T_R(I)) \cong \Hom_{\pi_*R}(\pi_*M, I)$;
\item[(c)] $\pi_*T_R(I) \cong I$;
\item[(d)] if an $R$-module $M$ is such that $\pi_*M \cong I$, then $M \simeq T_R(I)$.
\end{itemize}
\end{lem}
\begin{proof}
This is an application of Brown representability, see for instance~\cite[\S 4.1]{BHV}.
\end{proof}

\begin{defn}
Given a commutative noetherian ring spectrum $R$ and $\p \in \Spech(\pi_*R)$, we write 
$\I_\p=T_R(I_\p)$ where $I_\p$ is the injective hull of $\pi_*R/\p$.
\end{defn}

Based on the classical local Gorenstein duality statement in algebra, Barthel-Heard-Valenzuela 
made a definition of local Gorenstein duality for ring spectra~\cite{BHV}, 
generalizing the Gorenstein duality definition of Dwyer-Greenlees-Iyengar~\cite{DGI} 
and the version observed in modular representation theory by Benson-Greenlees~\cite{BGduality}. 
Here we give a slight modification of their definition.
\begin{defn}
  Let $R$ be a commutative noetherian ring spectrum. We say that 
  $R$ has \emph{local Gorenstein duality} if there exists a function 
  $\nu\colon \Spech(\pi_*R) \to \Z$ and an equivalence 
  \[
  \Gamma_\p R_\p \simeq \Sigma^{\nu(\p)} \I_\p
  \] 
  for all $\p \in \Spech(\pi_*R)$. We call $\nu$ the \emph{shift function} of $R$.
\end{defn}

\begin{rem}\label{rem:compare}
The previous definition is a slight generalization of the definition given in~\cite[4.21]{BHV}, 
where the shift function was of the form $\nu(\p) = a + \mrm{coht}(\p)$ for a fixed $a$, for each prime $\p$. 
In the cited paper, the authors call the local Gorenstein duality property, the absolute Gorenstein condition. 
In~\cite{BCHV} they renamed this local Gorenstein duality.
\end{rem}

We now recall some consequences of local Gorenstein duality on the homotopy groups from Greenlees-Lyubeznik~\cite{GL}. 
The following result shows that if $R$ satisfies local Gorenstein duality, then $\pi_*(R_\p) = (\pi_*R)_\p$ is a ring with a local 
cohomology theorem in the sense of~\cite{GL}, for each $\p \in \Spech(\pi_*R)$. 
\begin{thm}\label{thm:LCT}
Let $R$ be a commutative noetherian ring spectrum which satisfies local Gorenstein duality of shift $\nu$. 
For each $\p \in \Spech(\pi_*R)$, there is a strongly convergent spectral sequence
 \[E^2_{s,t} = H_\p^{-s}((\pi_*R)_\p)_t \Rightarrow (I_\p)_{s+t-\nu(\p)}.\]
\end{thm}
\begin{proof}
Taking $M = R_\p$ in the spectral sequence (\ref{eq:lcss}), we obtain a strongly convergent spectral sequence 
\[E^2_{s,t} = H_\p^{-s}((\pi_*R)_\p)_t \Rightarrow \pi_{s+t}(\Gamma_\p R_\p).\]
Since $R$ is assumed to satisfy local Gorenstein duality of shift $\nu$, we have 
$\Gamma_\p R_\p \simeq \Sigma^{\nu(\p)} \I_\p$ and the claim then follows by Lemma~\ref{lem:matlis}(c).
\end{proof}

By Grothendieck vanishing, the spectral sequence in the previous theorem collapses surprisingly often, yielding many structural implications for $\pi_*R$. We refer the reader to~\cite{GL} for more consequences and details and we satisfy ourselves with recalling two of the most striking.
\begin{prop}[{\cite[5.3, 7.4]{GL}}]\label{prop:coeffconsequences}
Let $R$ be a commutative noetherian ring spectrum which has local Gorenstein duality. 
\begin{itemize}
\item[(a)] Then $\pi_*R$ is Gorenstein if and only if it is Cohen-Macaulay.
\item[(b)] Moreover, $\pi_*R$ is generically Gorenstein; that is, each localization at a minimal prime ideal is Gorenstein. 
\end{itemize}
\end{prop}

\subsection{Proving local Gorenstein duality}
We now recall the key methods of proving when ring spectra satisfy local Gorenstein duality. In particular, we recall the ascent theorem for local Gorenstein duality from~\cite{BHV}. In Section~\ref{sec:descent} we will prove a descent theorem for local Gorenstein duality.

\begin{defn}
  We say that a commutative noetherian ring spectrum $R$ is
  \emph{algebraically Gorenstein with shift function $\nu$} if $\pi_*R$ is a graded 
  Gorenstein ring with shift function $\nu$. 
\end{defn}

\begin{lem}\label{lem:alggor}
If $R$ is algebraically Gorenstein with shift function $\nu$, then $R$ has local Gorenstein duality with shift function defined by $\sigma(\p) = \nu(\p) - \mrm{ht}(\p)$.
\end{lem}

\begin{proof}
The strongly convergent spectral sequence~(\ref{eq:lcss}) with $M=R_\p$
\[E^2_{s,t} = H_\p^{-s}(\pi_*R_\p)_t \Rightarrow \pi_{s+t}(\Gamma_\p R_\p)\] collapses to show that $R$ has local Gorenstein duality 
with shift function $\sigma(\p) = \nu(\p) - \mrm{ht}(\p)$. 
\end{proof}

\begin{defn}
Let $f\colon R \to S$ be a map of commutative ring spectra. 
\begin{itemize}
\item We say that $f$ is \emph{finite} if $S$ is a compact $R$-module.
\item We say that $f$ is \emph{relatively Gorenstein of shift $a$} (for some $a \in \Z$) if there is 
an equivalence $\Hom_R(S,R) \simeq \Sigma^a S$ of $S$-modules.
\end{itemize}
\end{defn}

Let us now give some interesting examples of finite and relatively Gorenstein maps. We will return to some of these examples in Section~\ref{sec:descent}.
\begin{ex}\label{ex-finite-rel-gor-map} \leavevmode
\begin{itemize}
\item[(i)] The complexification map $f\colon KO \to KU$ is finite by Wood's theorem and relatively Gorenstein of shift $-2$ by~\cite[19.1]{Greenlees2018}. The same holds if we replace the $K$-theory spectra with their connective counterparts. 
\item[(ii)] There is a ring map $f\colon \mrm{tmf}_{(3)} \to \mrm{tmf}_1(2)_{(3)}$ to the ring of topological modular forms with indicated level structure. Mathew~\cite[4.15]{Mathewtmf} showed that there is a finite complex $X_3$ and an equivalence of $\mrm{tmf}_{(3)}$-modules $\mrm{tmf}_{(3)} \wedge X_3 \simeq \mrm{tmf}_1(2)_{(3)}$, showing that $f$ is finite. Using this Greenlees~\cite[19.2]{Greenlees2018} showed that $f$ is relatively Gorenstein of shift $-8$. 
\item[(iii)] In a similar way, there is a ring map $f\colon \mrm{tmf}_{(2)} \to \mrm{tmf}_1(3)_{(2)}$. 
Mathew~\cite[4.12]{Mathewtmf} showed that there is a finite complex $DA(1)$ and an equivalence of $\mrm{tmf}_{(2)}$-modules $\mrm{tmf}_{(2)} \wedge DA(1)\simeq \mrm{tmf}_1(3)_{(2)}$, showing that $f$ is finite. Greenlees then showed that $f$ is relatively Gorenstein of shift $-12$~\cite[19.2]{Greenlees2018}. 
\end{itemize}
\end{ex}

Given a map $f\colon R \to S$ of commutative ring spectra, we write $f^*\colon \mod{S} \to \mod{R}$ for the restriction of scalars functor. The restriction of scalars $f^*$ has a left adjoint given by the extension of scalars functor $f_* = S \otimes_R -$, and a right adjoint given by the coextension of scalars functor $f_! = \Hom_R(S,-)$.

There is the following ascent result for relatively Gorenstein maps.
\begin{lem}\label{lem:relGorcomposite}
Let $f\colon R \to S$ and $g\colon S \to T$ be maps of commutative ring spectra and suppose that $f\colon R \to S$ is relatively Gorenstein of shift $a$. 
Then $gf\colon R \to T$ is relatively Gorenstein if and only if $g\colon S \to T$ is relatively Gorenstein with shifts related by $\mathrm{shift}(gf) = \mathrm{shift}(g) + a.$
\end{lem}
\begin{proof}
By the restriction-coextension of scalars adjunction, we have
\[\Hom_R(T,R) \simeq \Hom_S(T, \Hom_R(S,R)) \simeq \Sigma^a\Hom_S(T,S)\]
and the result follows from this.
\end{proof}

\begin{lem}\label{lem:ambi}
Let $f\colon R \to S$ be a finite map of commutative ring spectra which is relatively Gorenstein of shift $a$. There is a natural equivalence $f_!M \simeq \Sigma^a f_*M$ for all $R$-modules $M$. 
\end{lem}
\begin{proof}
If $f\colon R \to S$ is relatively Gorenstein of shift $a$, and $S$ is compact as an $R$-module, we have \[f_!M = \Hom_R(S, M) \simeq \Hom_R(S,R) \otimes_R M \simeq \Sigma^a S \otimes_R M = \Sigma^a f_*M\] which gives the claim.
\end{proof}

Given a map $f\colon R \to S$ of commutative noetherian ring spectra, we write $\res_f\colon \Spech(\pi_*S) \to \Spech(\pi_*R)$ for the induced map. 
The following is the ascent theorem for local Gorenstein duality. In Section~\ref{sec:descent} we will prove a descent theorem for local Gorenstein duality.

\begin{thm}[{\cite[4.27, 4.32]{BHV}}]\label{thm:Gorascent} 
  Let $f\colon R \to S$ be a finite map of commutative noetherian ring spectra, 
  which is relatively Gorenstein of shift $a\in\Z$. 
  If $R$ has local Gorenstein duality of shift $\nu$, then $S$ has local Gorenstein 
  duality with shift function $\sigma$ given by $\sigma(\q) = \nu(\res_f(\q)) - a$ 
  for all $\q\in\Spech(\pi_*(S))$.
\end{thm}

\begin{proof}
We give a sketch of the proof due to Barthel-Heard-Valenzuela~\cite[4.27]{BHV} since our conventions for Gorenstein shifts are different, and refer the reader there for more details. Fix a prime ideal $\p$ of $\pi_*R$ and write $U$ for the preimage of $\p$ under $\res_f$. Then one has decompositions 
\[
f_!(\I_\p) \simeq \bigoplus_{\q' \in U}\I_{\q'} \qquad \mathrm{and} \qquad f_*(\Gamma_\p R_\p)= \bigoplus_{\q' \in U}\Gamma_{\q'} S_{\q'}.
\]
For $\q \in U$, we have \[
\Gamma_\q S_\q \simeq  \Gamma_\q S_\q \otimes_S \bigoplus_{\q' \in U}\Gamma_{\q'} S_{\q'} \simeq \Gamma_\q S_\q \otimes_S f_*(\Gamma_\p R_\p) 
\simeq \Sigma^{-a} \Gamma_\q S_\q \otimes_S f_!(\Gamma_\p R_\p)
\] 
using a support argument and the relative Gorenstein condition. Since $R$ has local Gorenstein duality we have 
\[f_!(\Gamma_\p R_\p) \simeq \Sigma^{\nu(\p)}f_!(\I_\p) \simeq \Sigma^{\nu(\p)}\bigoplus_{\q' \in U} \I_{\q'}.\]
Combining these two equivalences we have
\[
\Gamma_\q S_\q \simeq \Sigma^{\nu(\p)-a}\Gamma_\q S_\q \otimes_S \bigoplus_{\q' \in U} \I_{\q'} \simeq \Sigma^{\nu(\p)-a}\I_\q
\] 
where the last equivalence follows since $\I_\q$ is $\q$-torsion and $\q$-local.
\end{proof}

\section{Venkov's theorem}
In this section we extend Venkov's theorem~\cite{Venkov} to complex oriented ring spectra. The special case where the ring $R$ is discrete and noetherian is well-known, see for instance~\cite[3.10.1]{Bensonbook} for a topological proof and~\cite[4.2.1]{Bensonbook} for an algebraic one. The case where the ring spectrum $R$ is complex orientable is probably also well-known but since we do not know of a good reference we decided to include the details. 

The key feature of complex oriented ring spectra which makes Venkov's theorem work 
in this generality is that the cohomology of $BU(n)$ is a power series ring.
\begin{lem}[{\cite[4.3.3]{Kochman}}]\label{lem:cx-orientable}
  Let $R$ be a complex orientable ring spectrum. 
  Then for all $n\geq 1$, we have $R^*(BU(n))=R^*\llbracket c_1,\ldots, c_n\rrbracket$ 
  on the Chern classes $c_i$ of degree $-2i$. 
\end{lem}

We recall the following from~\cite[10.2, 10.5]{Greenlees2018}, see 
also~\cite[2.1]{Mathew2015}.

\begin{lem}\label{lem-regular-compact}
  Let $R$ be a commutative ring spectrum and consider an $R$-module $M$. 
  \begin{itemize}
  \item[(a)] If $\pi_*R$ is noetherian and $M$ is a compact $R$-module, 
  then $\pi_*M$ is a finitely generated $\pi_*R$-module.
  \item[(b)] If $\pi_*R$ has finite global dimension and $\pi_*M$ is finitely generated 
  over $\pi_*R$, then $M$ is a compact $R$-module.
  \end{itemize}
\end{lem}

The proof of the following lemma is essentially the same as that of~\cite[3.10.1]{Bensonbook}.

\begin{lem}\label{lem:venkov1}
  Let $R$ be a complex orientable commutative noetherian ring spectrum. 
  Consider a compact Lie group $G$ and a faithful 
  representation $i\colon G \to U(n)$. 
  Then $R^*(BG)$ is a finitely generated $R^*(BU(n))$-module.  
\end{lem}

\begin{proof}
  Consider the following diagram
  \[
  \begin{tikzcd}
  * \arrow[r] & BU(n) \arrow[r, "1"] & BU(n)\\
  U(n)/G \arrow[u] \arrow[r] & BG \arrow[r, "Bi"] \arrow[u, "Bi"] 
  & BU(n) \arrow[u, "1"]
  \end{tikzcd}
  \]
  whose rows are fibre sequences.
  We apply the Atiyah–Hirzebruch spectral sequence on each row 
  to obtain a morphism of spectral sequences
  \[
  \begin{tikzcd}
  E_2^{*,*}=H^*(BU(n); R^*) \arrow[d]\arrow[Rightarrow, r] & R^*(BU(n)) 
  \arrow[d] \\
  E_2^{*,*}=H^*(BU(n); R^*(U(n)/G)) \arrow[Rightarrow, r] & R^*(BG).
  \end{tikzcd}
  \]
  Since $BU(n)$ admits the structure of a finite CW-complex, the 
  spectral sequences converge strongly to the target. 
  Note also that $H^*(BU(n); R^*)=R^*\llbracket c_1, \ldots, c_n\rrbracket= R^*(BU(n))$ 
  by Lemma~\ref{lem:cx-orientable} so the lower spectral sequence 
  is a spectral sequence of $R^*(BU(n))$-modules. 
  
  Note that $U(n)/G$ admits the structure of a finite CW-complex. 
  Thus by Lemma~\ref{lem-regular-compact}(a), we know that $R^*(U(n)/G)$ is 
  a finitely generated $R^*$-module.
  
  The $E_2$-page of the spectral sequence is
  \[
  H^*(BU(n); R^*(U(n)/G))= R^*(U(n)/G) \otimes_{R^*} H^*(BU(n); R^*)=
  R^*(U(n)/G)\llbracket c_1,\ldots, c_n\rrbracket
  \]
  so by the discussion in the previous paragraph we see that 
  the $E_2$-page is a finitely generated $R^*(BU(n))$-module. 
  
  The $E_\infty$-page is a sub-quotient of the $E_2$-page and 
  $R^*(BU(n))$ is noetherian by the Hilbert basis theorem. It 
  follows that $E_\infty$ is a finitely generated $R^*(BU(n))$-module. 
  Therefore $R^*(BG)$ has a finite filtration of finitely generated 
  $R^*(BU(n))$-modules so is itself finitely generated over 
  $R^*(BU(n))$.
\end{proof}

\begin{thm}[Venkov]\label{thm:venkov2}
  Let $R$ be a complex orientable commutative noetherian ring spectrum. 
  Consider a compact Lie group $G$ and a subgroup 
  $H\leq G$. Then $R^*(BH)$ is a finitely generated $R^*(BG)$-module.  
\end{thm}
\begin{proof}
 Choose a faithful representation $G \to U(n)$ and consider the 
 composite $H \to G \to U(n)$. We obtain restriction maps
 in cohomology
 \[
 R^*(BU(n)) \to R^*(BG) \to R^*(BH).
 \]
 Since $R^*(BH)$ is finitely generated as a $R^*(BU(n))$-module, it is 
 also finitely generated as a $R^*(BG)$-module.
\end{proof}

\begin{cor}\label{cor:noetherian}
  Let $R$ be a complex orientable commutative noetherian ring spectrum, and let $G$ be a compact Lie group. Then $R^*(BG)$ is noetherian.
\end{cor}
\begin{proof}
By Venkov's theorem we know that $R^*(BG)$ is a finitely generated $R^*(BU(n))$-module for some $n$. Since $R^*(BU(n))$ is a power series ring, it follows that $R^*(BG)$ is a finitely generated $R^*$-algebra, and hence is noetherian.
\end{proof}

\section{Borel equivariant spectra}
 
 In this section we recall some equivariant preliminaries. In particular, we recall the connection between 
 $G$-equivariant Borel spectra and spectra with a $G$-action, and relate Borel spectra to cochain spectra. We will also discuss some orientability conditions which we will need later on. 
 
 Throughout we let $G$ be a compact Lie group and let $\Sp_G$ denote the $\infty$-category 
 of genuine $G$-spectra. Recall that $EG$ denotes the universal free $G$-space, and that a $G$-spectrum $X$ is said to be
 cofree (or Borel complete) if the natural map $X \to \Hom (EG_+, X)$ is 
 an equivalence in $\Sp_G$. 
 The full subcategory of $\Sp_G$ spanned by the cofree $G$-spectra is denoted by 
 $\Sp_G^{\mathrm{cofree}}$.

\begin{defn}
 Given a spectrum $X$, we define its associated \emph{Borel $G$-spectrum} by
 $b_GX = \Hom(EG_+,X)$ where $X$ is implicitly inflated to a $G$-spectrum.
 By adjunction this is a cofree $G$-spectrum. If $R$ is 
 a commutative ring spectrum, then the Borel spectrum $b_GR$ is a commutative ring 
 $G$-spectrum via the diagonal on $EG_+$.
\end{defn}
 
 A spectrum with a $G$-action 
 is an object in the $\infty$-category $\Fun(BG,\Sp)$. The relation between these two 
 notions is explained by the following result. 
 
\begin{lem}
  Let $G$ be a compact Lie group and $X$ be a $G$-spectrum. Write $X_u$ for 
  the underlying spectrum with its residual $G$-action. The 
  underlying functor $(-)_u$ descends to a symmetric monoidal 
  equivalence of $\infty$-categories
  $\Sp_G^{\mathrm{cofree}}\simeq \Fun(BG, \Sp)$.
\end{lem}
\begin{proof}
 See~\cite[6.17]{MathewNaumannNoel17} for a proof for finite groups, and~\cite[II.2.7]{NikolausScholze} for the general case.
\end{proof}

 The following result connects the category of modules over 
 the Borel spectrum and the category of representations of 
 non-equivariant modules. Since any compact 
 $b_G R$-module is necessarily cofree, the previous lemma gives the following. 

\begin{prop}[{\cite[6.21]{MathewNaumannNoel17}}]\label{prop-borel}
 Let $G$ be a compact Lie group and $R$ be a commutative ring spectrum. Then the 
 functor 
 \[
 (-)_u \colon \mod{b_G R} \to \Fun(BG, \mod{R})
 \]
 is fully faithful when restricted to compact objects.
\end{prop}

\subsection{Borel spectra and cochains}
Let $H$ be a closed subgroup of a compact Lie group $G$. There is a natural restriction 
functor $\mrm{res}^G_H\colon \Sp_G \to \Sp_H$ and this has a left adjoint $\mrm{ind}_H^G = G_+ \wedge_H -$ and a right adjoint $\mrm{coind}_H^G = F_H(G_+, -)$ called induction and coinduction respectively. These satisfy the following important properties which we will use throughout.
\begin{thm}
Let $G$ be a compact Lie group, $H \leq G$ be a closed subgroup, $X$ be a $G$-spectrum and $Y$ be an $H$-spectrum.
\begin{itemize}
\item[(a)] \emph{Wirthm\"uller isomorphism}: There is a natural equivalence \[\mrm{ind}_H^G(S^{-L(G,H)} \wedge Y) \simeq \mrm{coind}_H^GY\] where the $H$-representation $L(G,H)$ is the tangent space at the identity coset $eH$ of the smooth manifold $G/H$, see also 
Definition~\ref{defn-representations}. 
\item[(b)] \emph{The (right) projection formula}: There is a natural equivalence \[\mrm{coind}_H^G(Y \wedge \mrm{res}_H^GX) \simeq \mrm{coind}_H^GY \wedge X.\]
\item[(c)] \emph{The (left) projection formula}: There is a natural equivalence \[\mrm{ind}_H^G(Y \wedge \mrm{res}_H^GX) \simeq \mrm{ind}_H^GY \wedge X.\]
\end{itemize}
\end{thm}
\begin{proof}
Part (a) may be found in~\cite[II.6.2]{LMS}. Parts (b) and (c) are~\cite[II.4.9]{LMS}. \end{proof}

The following lemma records the relation between Borel spectra and change of groups 
functors.

\begin{lem}\label{lem:omnibusBorel}
Let $G$ be a compact Lie group and consider a closed subgroup $H\leq G$. 
\begin{itemize}
\item[(a)] For any $G$-spectrum $X$, there is a natural equivalence $\mrm{coind}_H^G\mrm{res}_H^GX \simeq DG/H_+ \wedge X$. In particular, we have $\pi_*^G(DG/H_+ \wedge X)= \pi_*^H(X)$.
\item[(b)] For any (non-equivariant) spectrum $X$, there are natural equivalences
$\mrm{res}^G_H b_GX \simeq b_HX$ and $\mrm{coind}_H^G b_H X\simeq DG/H_+ \wedge b_G X$.
\end{itemize}
\end{lem}
\begin{proof}
The first part of (a) follows from the projection formula together with the fact that $\mrm{coind}_H^G(S^0) \simeq DG/H_+$. The second part of (a) then follows from this by applying the restriction-coinduction adjunction. 

For part (b), let $Y$ be a $G$-spectrum. We first claim that there is an 
equivalence $\mrm{res}^G_H \Hom (EG_+,Y) \simeq \Hom (EH_+,\mrm{res}^G_HY)$. 
Indeed for any $H$-spectrum $Z$, we have equivalences
\begin{align*}
\Map_{\Sp_H}(Z, \mrm{res}^G_H \Hom (EG_+, Y)) &\simeq \Map_{\Sp_G}(EG_+ \wedge \mrm{ind}^G_HZ, Y) \\
&\simeq \Map_{\Sp_G}(\mrm{ind}^G_H (EH_+ \wedge Z), Y) \\
&\simeq \Map_{\Sp_H}(Z, \Hom (EH_+,\mrm{res}^G_HY))
\end{align*}
where the second isomorphism follows from the fact that $\mrm{res}^G_H EG_+ \simeq EH_+$, combined with the projection formula. Hence the claim follows from the Yoneda lemma. 
The first part of (b) now follows from this claim and the fact that $\mrm{res}^G_H \mrm{inf}^G_1 \simeq \mrm{inf}^H_1.$ The second equivalence in (b) then follows from combining the first part of (b) with (a).
\end{proof}

The following lemma shows the relation between Borel equivariant spectra and cochain spectra.

\begin{lem}\label{lem:fixedpointscoind}
Let $G$ be a compact Lie group, $H \leq G$ be a closed subgroup, and let $R$ be a commutative ring spectrum.
\begin{itemize}
\item[(a)] There is an equivalence $(b_GR)^G \simeq C^*(BG;R)$ of ring spectra.
\item[(b)] There is an equivalence $(DG/H_+\wedge b_GR)^G \simeq C^*(BH;R)$ of $C^*(BG;R)$-modules.
\item[(c)] There is an equivalence \[\Hom_{b_GR}(b_GR\wedge DG/H_+, b_GR)^G \simeq C^*(BH^{-L(G,H)};R)\] of $C^*(BG;R)$-modules where  $BH^{-L(G,H)}=EH_+ \wedge_H S^{-L(G,H)}$ denotes the Thom spectrum.
\end{itemize}
\end{lem}
\begin{proof}

The claims in parts (a) and (b) follow by adjunction arguments together with part (b) of Lemma~\ref{lem:omnibusBorel}.
For part (c), note that $DG/H_+ \simeq G_+ \wedge_H S^{-L(G,H)}$ by the Wirthm\"{u}ller isomorphism. Therefore we get
\[
\Hom_{b_GR}(b_GR\wedge DG/H_+, b_GR)^G \simeq \Hom(G_+ \wedge_H S^{-L(G,H)}, b_GR)^G.\]
By adjunction and the projection formula, 
\begin{align*} \Hom(G_+ \wedge_H S^{-L(G,H)}, b_GR)^G &\simeq \Hom(G_+ \wedge_H (EH_+ \wedge S^{-L(G,H)}), R)^G \\ &\simeq \Hom(EH_+ \wedge S^{-L(G,H)}, R)^H\end{align*} and the claim then follows since $R$ is inflated from a non-equivariant spectrum.
\end{proof}

\subsection{Orientability}\label{sec:orientations}
We finish this section by discussing a certain notion of orientability on representations. 
Firstly we recall some important representations, more details 
can be found for instance in~\cite[3.2.1]{Schwede}.

\begin{defn}\label{defn-representations}
 Consider a closed subgroup $H$ of a compact Lie group $G$. 
 We consider the following representations:
 \begin{itemize}
 \item the adjoint representations $LG=T_e G$ and $LH=T_e H$. These are the tangent 
 space at $e$ of the smooth manifolds $G$ and $H$ respectively, endowed with the 
 conjugation action. 
 \item the tangent $H$-representation $L(G,H)=T_{eH}G/H$. 
 This is the tangent space at $eH$ of the smooth manifold $G/H$ which admits an 
 $H$-action by left translation.  If $H$ has 
 finite index in $G$, then $L(G,H)$ is trivial.
 \item the representation $\Ad(G,H)= LG/LH$ which has an action by conjugation 
 of the normalizer group $N_G(H)$. 
 In general $H\leq N_G(H)$ does not act trivially on $\Ad(G,H)$.
 \end{itemize}
 The differential of the projection map $G \to G/H$ induces an isomorphism 
 $\Ad(G,H) \simeq L(G,H)$. One checks that the conjugation action of $H\leq N_G(H)$ 
 on $\Ad(G,H)$ corresponds to the left translation action of $H$ on $L(G,H)$. 
 Therefore we have an isomorphism of $H$-representations $\Ad(G,H)\simeq L(G,H)$.
\end{defn}

We now recall the definition of orientations, and give some consequences which we will require; see~\cite[\S 3]{orientations} for more details.

\begin{defn}\label{def-cx stable}
 Let $R$ be a commutative ring $G$-spectrum and let $V$ be a $G$-representation 
 of dimension $d$. An $R$-\emph{orientation} $u_V$ for $V$ is a $G$-equivariant 
 map $u_V \colon S^{d}\to R \wedge S^V $ such that the composite 
 \[
 u_V^R \colon R \wedge S^{d} \xrightarrow{R \wedge u_V } R \wedge R \wedge S^{V} 
 \xrightarrow{\mu \wedge S^{V}} R \wedge S^{V} 
 \]
 is an equivariant equivalence. Note that given an equivariant equivalence $u_V^R \colon  R \wedge S^{d} \to R \wedge S^{V}$ 
 which is also an $R$-module map, we can precompose $u_V^R$ with the map 
 $\eta \wedge S^{d} \colon S^0 \wedge S^{d} \to R \wedge S^{d}$ to obtain an 
 $R$-orientation.
\end{defn}

Let us give a trivial but important example. 

\begin{ex}
 If $G$ is a finite group, then $LG$ is zero and so $R$-orientable for any commutative ring $G$-spectrum $R$.
\end{ex}

By construction, weak equivalences between cofree spectra are detected on the underlying, non-equivariant homotopy. This gives the following way to detect 
$R$-orientability.

\begin{lem}[{\cite[3.5]{orientations}}]\label{lem-cofree-orientation}
Let $G$ be a compact Lie group, $R$ be a commutative ring $G$-spectrum and let $V$ be a $G$-representation of dimension $d$. 
If $R$ is cofree and there is an element $u_V \in \pi_d^G(R\wedge S^V)$ such that $u_V^R$ induces a non-equivariant equivalence, then $V$ is $R$-orientable.
\end{lem}

We now give some examples of representations which are Borel orientable.
\begin{lem}\label{lem-representation-borel-oriented}
Let $G$ be a compact Lie group, $V$ a complex $G$-representation and let $R$ be a complex oriented commutative ring spectrum. Then $V$ admits a $b_GR$-orientation. 
\end{lem}
\begin{proof}
 Since $R$ is complex orientable there is an equivalence 
 $\Sigma^{|V|}R \simeq \Sigma^{V} R$. 
 Applying the functor $\Hom(EG_+, -)$ gives an equivariant equivalence 
 $b_GR \wedge S^{|V|} \simeq b_GR \wedge S^{V}$. 
\end{proof}

\begin{ex}
 Let $R$ be a complex oriented ring spectrum and let $G$ be a complex compact Lie 
 group (e.g., a complex torus). Then $LG$ is a complex representation and hence 
 it is $b_GR$-orientable. 
\end{ex}

We now recall the following definition from~\cite{Greenlees2020} and an important consequence.
\begin{defn}\label{defn:connected}
 Let $k$ be a field and consider a compact Lie group $G$. 
 We say that $G$ is \emph{$k$-connected} if either $G$ is connected or if there is a prime 
 $p$ and a positive integer $n$ such that $p^n k=0$ and $\pi_0(G)$ is a $p$-group.
\end{defn}

\begin{lem}\label{lem-k-connected}
 Let $k$ be a field and let $G$ be a compact Lie group which is $k$-connected. 
 Then every $G$-representation is $b_Gk$-orientable.
\end{lem}

\begin{proof}
 Let $V$ be a $G$-representation of dimension $d$ and consider the 
 homotopy fixed points spectral sequence
 \[
 E_2^{s,t}=H^{s}(BG; H_t(S^V;k)) \Rightarrow \pi_{t-s}^G(b_G k \wedge S^V).
 \]
 The $G$-action on $H_*(S^V;k)$ factors through $\pi_0(G)$ and this acts 
 trivially as there are no nonzero homomorphisms 
 $\pi_0(G) \to \mrm{Aut}_k(H_*(S^V;k))=k^\times$ by $k$-connectedness.
 Therefore the system of local coefficients is simple.  
 As the spectral sequence is concentrated in the line $t=d$, 
 the element 
 \[
 1\in k=H_d(S^V;k)= H^0(BG; H_d(S^V;k))
 \]
 is a permanent cycle and so defines an element $u_V \in \pi_d^G(b_G k \wedge S^V)$. 
 Base changing this we get a map $ b_G k\wedge S^d \to b_G k \wedge S^V$ which is
 a non-equivariant equivalence and so defines an orientation by 
 Lemma~\ref{lem-cofree-orientation}. 
\end{proof}

We now give a partial extension of the previous lemma for even commutative ring spectra. Later on
we will use this ingredient in verifying the duality property for cochain spectra.
\begin{lem}\label{lem:LUorientable}
 Let $R$ be a commutative ring spectrum with $\pi_*R$ even, and let $U$ be a unitary group.
 Then the representation $LU$ is $b_U R$-orientable.
\end{lem}
\begin{proof}
 Let us write $V=LU$ and $d$ for the dimension of $V$.
 Consider the homotopy fixed points spectral sequence
 \[
 E_2^{s,t}=H^{s}(BU; \pi_t(R \wedge S^V)) \Rightarrow \pi_{t-s}^U(b_U R \wedge S^V).
 \]
 Note that $\pi_0(U)$ acts trivially on $\pi_t(R \wedge S^V)$ since $U$ is connected. 
 Consider the element
 \[
 1 \in \pi_0(R)\simeq \pi_d(R \wedge S^V)= H^0(BU;\pi_d(R \wedge S^V)).
 \]
 We claim that this is a permanent cycle and a similar argument as in 
 Lemma~\ref{lem-k-connected} then gives the proposition. To prove the 
 claim note that $H^*(U; \Z)$ is polynomial on the Chern classes so it is nonzero 
 only in non-negative even degrees. Therefore we have a first and fourth quadrant spectral 
 sequence with zeroes in odd columns. If $d$ is odd (resp., even), 
 then the spectral sequence is also zero in even (resp., odd) rows since $R$ is concentrated in even degrees.
 The differentials on the $E_r$-page are of the form 
 $d_r \colon E_r^{s,t}\to E_r^{s+r,t+r-1}$ for all $r\geq 2$. 
 The claim follows by observing that on each page, there are no nonzero differentials 
 starting at or arriving into $(0,d)$.
\end{proof}

\begin{prop}\label{prop:LULGorientable}
Let $R$ be a commutative ring spectrum with $\pi_*R$ even, and $G$ be a compact Lie group.
Choose an embedding $i\colon G \to U$ into a unitary group.
If $LG$ is $b_GR$-orientable, then $L(U,G)$ is $b_G R$-orientable.
\end{prop}

\begin{proof}
 Recall that we have an isomorphism of $G$-representations $\Ad(U,G)\simeq L(U,G)$ 
 so it suffices to show that $\Ad(U,G)$ is $b_GR$-orientable. Note that 
 we have a cofibre sequence of $b_GR$-modules
 \[
 b_GR \wedge S^{LG} \to b_GR \wedge S^{i^* LU} \to b_GR \wedge S^{\Ad(U,G)}
 \]
 and that a choice of a $b_UR$-orientation for $LU$ induces  
 a $b_GR$-orientation for $i^*LU$ by restriction. The claim then follows from the 
 cofibre sequence and the assumption that $LG$ is $b_GR$-orientable, since $LU$ is $b_UR$-orientable
 by Lemma~\ref{lem:LUorientable}.
\end{proof}

\section{Generating modules over Borel spectra}
The goal of this section is to prove that for special choices of groups $G$, the category of $b_G R$-modules is generated by its unit whenever $R$ is complex orientable. This is a special feature of Borel cohomology since a priori the module category is generated by $G/H_+ \wedge b_G R$ for all closed subgroups $H\leq G$. 
The idea of the proof is heavily influenced by the argument given by Greenlees~\cite[\S 5]{Greenlees2020} in the case when $R$ is a discrete commutative ring. 

\begin{lem}\label{lem:torus}
Let $R$ be a complex orientable commutative ring spectrum and let $T$ be a torus. Then the $T$-spectrum $b_TR$ generates its category of modules.
\end{lem}
\begin{proof}
We can just follow the argument given in~\cite[5.1]{Greenlees2020}. The crucial ingredient is that Borel cohomology $b_T R$ has Thom isomorphisms as it is complex orientable, see Lemma~\ref{lem-representation-borel-oriented}. 
\end{proof}

Recall our use of terminology of building from the conventions section. 
\begin{thm}\label{thm:generates}
Let $G$ be a compact Lie group, $T$ be a maximal torus for $G$ and 
$R$ be a complex orientable commutative ring spectrum. 
If $b_GR$ and $DG/T_+ \wedge b_GR$ build each other, then $b_GR$ generates its category of modules.
\end{thm}
\begin{proof}
Let $M$ be a $b_GR$-module and suppose that $\pi_*^GM = 0$. Since $b_GR$ builds $DG/T_+ \wedge b_GR$, applying $- \otimes_{b_GR} M$ shows that $M$ builds $DG/T_+ \wedge M$. Using Lemma~\ref{lem:omnibusBorel}(a) and the assumption $\pi_*^G M=0$, we deduce that $\pi_*^TM = 0$, and hence $\pi_*^KM = 0$ for all $K \leq T$ by Lemma~\ref{lem:torus}. 

Consider a closed subgroup $H \leq G$. By the double coset formula together with Lemma~\ref{lem:omnibusBorel}(a) and the previous paragraph, we have
\[\pi_*^H(DG/T_+ \wedge M) = \prod_{g \in T \backslash G/ H} \pi_*^{T \cap H^g}(M)=0.\]
Therefore $DG/T_+ \wedge M \simeq 0$, and since this builds $M$, we must also have $M \simeq 0$ as required.
\end{proof}

Theorem~\ref{thm:generates} becomes powerful when coupled with the following lemma and unipotence results of Mathew-Naumann-Noel~\cite{MathewNaumannNoel17}.

\begin{lem}\label{lem-flag-variety}
  Let $G$ be a compact Lie group, $H$ be a subgroup of $G$ and $R$ be a complex orientable commutative ring spectrum. The following are equivalent:
\begin{itemize}
  \item[(a)] $DG/H_+ \wedge R$ is free over $R$ as objects 
  of $\Fun(BG, \mod{R})$.
  \item[(b)] $DG/H_+ \wedge b_GR$ is free over $b_GR$ as objects of $\mod{b_GR}$.
  \end{itemize}
\end{lem}
\begin{proof}
By Proposition~\ref{prop-borel}, on compact $b_GR$-modules the functor $(-)_u$ is an equivalence onto its essential image. Since both $b_GR$ and $DG/H_+ \wedge b_GR$ are compact $b_GR$-modules the result follows, as $(-)_u$ sends $b_GR$ and $DG/H_+ \wedge b_GR$ to $R$ and $DG/H_+\wedge R$ respectively by Lemma~\ref{lem:omnibusBorel}(b).
\end{proof}

Combining the previous results, we obtain the following which is a crucial ingredient in our proof of local Gorenstein duality for cochain spectra.
\begin{thm}\label{thm:unitary}
Let $R$ be a complex orientable commutative ring spectrum and $U$ be a unitary group. Then $b_UR$ generates its category of modules.
\end{thm}
\begin{proof}
In~\cite[7.49]{MathewNaumannNoel17} it is proved that $DU/T_+ \wedge R \simeq \bigoplus_{i=1}^{n!} \Sigma^{k(i)} R$ as objects of $\Fun(BU, \mod{R})$, where $n = \mathrm{rank}(U)$. The result then follows by applying Theorem~\ref{thm:generates} and Lemma~\ref{lem-flag-variety}.
\end{proof}

\begin{rem}
One can also prove directly that $DU/T_+ \wedge b_UR$ is free over $b_UR$, see the discussion after~\cite[7.49]{MathewNaumannNoel17}. 
\end{rem}

Although we are most interested in the case of unitary groups for the purpose of this paper, we note that one can also deduce the following more general result for the complex $K$-theory spectrum $KU$.
\begin{cor}
Let $G$ be a connected compact Lie group with $\pi_1G$ torsion free. Then $b_G(KU)$ generates its category of modules.
\end{cor}
\begin{proof}
By~\cite[8.17]{MathewNaumannNoel17}, $DG/T_+ \wedge KU$ is free over $KU$ in $\mrm{Fun}(BG, \mod{KU})$. Applying Theorem~\ref{thm:generates} and Lemma~\ref{lem-flag-variety} then gives the result.
\end{proof}

In~\cite{Greenlees2020}, Greenlees proves that if $R$ is an ordinary commutative ring, then $b_GR$ generates its category of modules if $G$ is $R$-connected (see Definition~\ref{defn:connected}) by comparison of the K\"unneth and Eilenberg-Moore spectral sequences. In light of this, it would be interesting to answer the following question.
\begin{que}\label{que:generates}
If $R$ is a complex orientable commutative ring spectrum and $G$ is a connected compact Lie group, does $b_GR$ generate its category of modules?
\end{que}

\section{Local Gorenstein duality for cochains}
In this section we will give conditions under which the canonical map 
$C^*(BG;R) \to C^*(BH;R)$ induced by the inclusion $H \leq G$ is relatively Gorenstein. 
We will then use this to prove local Gorenstein duality for $C^*(BG;R)$. 

Firstly we recall the following version of~\cite[3.3(i)]{Greenlees2020}.
\begin{lem}\label{lem:naturalmap}
Let $G$ be a compact Lie group, $R$ be a commutative ring spectrum and $M, N \in \mod{b_GR}$. The natural map induced 
by taking $G$-fixed points
\[\Hom_{b_GR}(M,N)^G \to \Hom_{C^*(BG;R)}(M^G, N^G)\]
is an equivalence if $M$ is built from $b_GR$.
\end{lem}
\begin{proof}
Recall that $(b_GR)^G = C^*(BG;R)$ so the map is well-defined. The subcategory of $b_GR$-modules $M$ for which the natural map is an equivalence is localizing and clearly contains $M = b_GR$.
\end{proof}

\begin{lem}\label{lem:relGor}
Let $G$ be a compact Lie group, $R$ be a complex orientable commutative ring spectrum, and choose an embedding $G \to U$ into a unitary group. There is an equivalence of $C^*(BG;R)$-modules 
\[
\Hom_{C^*(BU;R)}(C^*(BG;R), C^*(BU;R)) \simeq C^*(BG^{-L(U,G)};R)
\]
where $L(U,G)$ is the tangent $G$-representation of $U/G$ from Definition~\ref{defn-representations}.
\end{lem}

\begin{proof}
Recall from Lemma~\ref{lem:fixedpointscoind} that 
\[
\Hom_{b_UR}(b_UR\wedge DU/G_+, b_UR)^U \simeq C^*(BG^{-L(U,G)};R)
\]
and that $(b_UR \wedge DU/G_+)^U\simeq C^*(BG;R)$. Therefore the equivalence in the statement can be rephrased as the natural map induced by taking $U$-fixed points
\[
\Hom_{b_UR}(b_UR \wedge DU/G_+, b_UR)^U \to \Hom_{C^*(BU;R)}(C^*(BG;R), C^*(BU;R))\] 
being an equivalence. 
By Lemma~\ref{lem:naturalmap}, it is enough to show that $b_UR \wedge DU/G_+$ is built from $b_U R$. This holds by Theorem~\ref{thm:unitary}.
\end{proof}

We note that even though we only know that $b_GR$ generates its category of modules for $G$ a torus or unitary group, we can deduce the relative Gorenstein condition more generally using the ascent property of relative Gorenstein maps.
\begin{thm}\label{thm:relGor2}
Let $G$ be a compact Lie group, $H \leq G$, and $R$ be a commutative ring spectrum 
with $\pi_*R$ even. If $LG$ is $b_GR$-orientable and $LH$ is $b_HR$-orientable, then there is an equivalence of $C^*(BH;R)$-modules 
\[
\Hom_{C^*(BG;R)}(C^*(BH;R), C^*(BG;R)) \simeq \Sigma^{\dim(G/H)}C^*(BH;R).
\]
In other words, the map $C^*(BG;R) \to C^*(BH;R)$ is relatively Gorenstein of shift $\mrm{dim}(G/H)$.
\end{thm}
\begin{proof}
Choose an embedding $G \hookrightarrow U$ into a unitary group. 
If $LG$ is $b_GR$-orientable, then $L(U,G)$ is $b_GR$-orientable by Proposition~\ref{prop:LULGorientable}, and similarly for $L(U,H)$. 
Since $\pi_*R$ is even, $R$ admits a complex orientation and so Lemma~\ref{lem:relGor} tells us that $C^*(BU;R)\to C^*(BG;R)$ and $C^*(BU;R)\to C^*(BH;R)$ are relatively Gorenstein of shift $\dim(U/G)$ and $\dim(U/H)$ respectively. To conclude apply Lemma~\ref{lem:relGorcomposite}.
\end{proof}

\begin{rem}
One can also prove a version of the previous theorem when $LG$ and $LH$ are not orientable, see~\cite[6.1]{Greenlees2020} for details of the case when $R$ is an ordinary commutative ring. Since this generalization is not required for our goal of proving local Gorenstein duality for $C^*(BG;R)$ (and since the argument is the same as in the case where $R$ is an ordinary commutative ring) we have chosen to omit the details here.
\end{rem}

Before proving our main theorem we need another lemma. 

\begin{lem}\label{lem:compact}
Let $R$ be a complex orientable commutative noetherian ring spectrum with $\pi_*R$ 
of finite global dimension. Let $G$ be a compact Lie group and choose a faithful representation $G \hookrightarrow U$ into a unitary group. Then $C^*(BG;R)$ is a compact $C^*(BU;R)$-module. 
\end{lem}

\begin{proof}
Note that $\pi_*C^*(BU;R)$ is a finite power series ring on $\pi_*R$ and so it has 
finite global dimension. Now combine Lemma~\ref{lem-regular-compact} with Venkov's theorem.
\end{proof}

We are finally ready to prove our main result. 
\begin{thm}\label{thm:duality}
Let $G$ be a compact Lie group and $R$ be a commutative noetherian ring spectrum with $\pi_*R$ even and of finite global dimension. 
If $LG$ is $b_GR$-orientable, then $C^*(BG;R)$ has local Gorenstein duality. Moreover the shift function on maximal ideals $\m$ of $R^*(BG)$ takes the form
\[
\sigma(\m) = \dim(G) +\nu(\m \cap \pi_*R) - \dim(\pi_*R_{\m \cap \pi_*R})
\]
where $\nu$ denotes the Gorenstein shift function of the graded ring $\pi_*R$. 
If moreover $\pi_*R$ is equicodimensional, then the shift function simplifies even further to
\[
\sigma(\m) = \dim(G) +\nu(\m \cap \pi_*R) - \dim(\pi_*R).
\]
\end{thm}

\begin{proof}
Note that $\pi_*(C^*(BG;R)) = R^*(BG)$ which is noetherian by Corollary~\ref{cor:noetherian}.
Choose a faithful representation $G \hookrightarrow U(n)$ into a unitary group.
Since $\pi_*R$ is even, we may choose a complex orientation on $R$ and calculate 
$\pi_*(C^*(BU(n);R))=(\pi_*R)\llbracket c_1,\ldots, c_n\rrbracket$ with $|c_i|=-2i$.  
This is an algebraically Gorenstein ring by Corollary~\ref{cor-poly-shift-gor}.
Hence $C^*(BU(n);R)$ has local Gorenstein duality by Lemma~\ref{lem:alggor}. 
Consider the map $f\colon C^*(BU(n);R) \to C^*(BG;R)$. This is finite by Lemma~\ref{lem:compact} and relatively Gorenstein of shift $n^2-\dim(G)$ by the proof of Theorem~\ref{thm:relGor2}. Therefore by local Gorenstein ascent (Theorem~\ref{thm:Gorascent}) applied to $f$ we deduce that $C^*(BG;R)$ has local Gorenstein duality as claimed.

Let us now track the shifts. We saw that $\pi_*(C^*(BU(n);R))=(\pi_*R)\llbracket c_1,\ldots, c_n\rrbracket$ with $|c_i|=-2i$. This is an algebraically Gorenstein ring and the shift function on maximal ideals takes the form $\n \mapsto n(n+1)+\nu(\n \cap \pi_*R)$ by Corollary~\ref{cor-poly-shift-gor}.
Hence $C^*(BU(n);R)$ has local Gorenstein duality with shift function $\n \mapsto n(n+1)+\nu(\n \cap \pi_*R)-\mrm{ht}(\n)$ on maximal ideals by Lemma~\ref{lem:alggor}. If $\n$ is maximal in $(\pi_*R)\llbracket c_1,\ldots, c_n\rrbracket$ then $\n = (\n \cap \pi_*R, c_1, \ldots, c_n)$ as in the proof of Corollary~\ref{cor-poly-shift-gor} so
\[
\mrm{ht}(\n) = \dim((\pi_*R)\llbracket c_1,\ldots, c_n\rrbracket_\n) = \dim(\pi_*R_{\n \cap \pi_*R}) + n.
\]
 Thus the shift function can be rewritten as $\n \mapsto n^2+\nu(\n \cap \pi_*R)-\dim(\pi_*R_{\n \cap \pi_*R}).$ 
 
 Now let $\m$ be a maximal ideal in $R^*(BG)$. The commutative diagram 
\[
\begin{tikzcd}
R \arrow[d] \arrow[dr] & \\
C^*(BU(n);R)\arrow[r,"f"'] & C^*(BG;R)
\end{tikzcd}
\]
shows that $\nu(\m \cap \pi_*R) = \nu(\res_f(\m) \cap \pi_*R)$. Note that 
$\res_f(\m)$ is again maximal since $R^*(BU(n)) \to R^*(BG)$ is finite (Lemma~\ref{lem:venkov1}) and $\mrm{coht}(\res_f(\m)) = \mrm{coht}(\m)=0$ (Lemma~\ref{lem:coheight}).   
The local Gorenstein ascent theorem applied to the map $f$ then shows that $C^*(BG;R)$ has local Gorenstein duality with shift function on maximal ideals given by
\[
\sigma(\m) = \dim(G) + \nu(\m \cap \pi_*R) -\dim(\pi_*R_{\m \cap \pi_*R}).
\]
This concludes the proof of the theorem.
\end{proof}

In the case when $R$ is a field, we recover previous results~\cite[4.32]{BHV} and~\cite[12.1]{BGduality}.

\begin{cor}\label{cor-LGD-field}
 Let $G$ be a compact Lie group and $k$ be a field. If $LG$ is $b_Gk$-orientable, then 
 $C^*(BG;k)$ has local Gorenstein duality with shift function $\sigma(\p) = \dim(G) + \mrm{coht}(\p)$ for all $\p\in \Spech(H^*(BG;k))$. 
\end{cor}

\begin{proof}
 We have already seen in Theorem~\ref{thm:duality} that $C^*(BG;k)$ has local Gorenstein duality so let us calculate 
 the shift function. We choose a faithful representation $G \hookrightarrow U(n)$ into a unitary 
 group. Note that $\pi_*(C^*(BG;k))=k\llbracket c_1, \ldots, c_n \rrbracket$ which is 
 algebraically Gorenstein with shift $\q \mapsto n(n+1)$ by 
 Corollary~\ref{cor-poly-shift-gor} and Theorem~\ref{thm:localize}. 
 Hence $C^*(BU(n);k)$ has local Gorenstein duality of shift 
 $\q \mapsto n(n+1)-\mrm{ht}(\q)$ by Lemma~\ref{lem:alggor}. By applying the dimension 
 formula (Lemma~\ref{lem:equi}) to $k\llbracket c_1,\ldots,c_n \rrbracket$ we can simplify the shift to 
 $\q \mapsto n(n+1)-n+\mrm{coht}(\q)=n^2+\mrm{coht}(\q)$. 
 We then apply local Gorenstein ascent (Theorem~\ref{thm:Gorascent}) to the map 
 $f\colon C^*(BU(n);k)\to C^*(BG;k)$ which is finite by Lemma~\ref{lem:compact} and 
 relatively Gorenstein of shift $n^2 -\dim(G)$ by Theorem~\ref{thm:relGor2}. 
 Putting all this together we see that $C^*(BG;k)$ has shift function $\sigma(\p)=\dim(G)+\mrm{coht}(\p)$ as claimed.
\end{proof}

\section{Examples}\label{sec:examples}

We will now apply Theorem~\ref{thm:duality} to various complex oriented 
ring spectra $R$. In the following list of examples we fix a compact Lie group $G$ of dimension $g$ 
and assume that $LG$ is $b_GR$-orientable for the appropriate $R$ in the example. We recall that this always holds if $G$ is finite,
and refer the reader to Section~\ref{sec:orientations} for more examples.
In each of the following examples we only give the shift functions on maximal ideals unless otherwise stated.

\begin{ex}[Ordinary cohomology]
 If $R=k$ is a field, then we deduce that $C^*(BG;k)$ has local Gorenstein 
 duality with shift function $\m \mapsto g$. In fact in this case the shift function 
 is given on all prime ideals by $\p \mapsto g+\mrm{coht}(\p)$ by 
 Corollary~\ref{cor-LGD-field} recovering~\cite[4.32]{BHV} and~\cite[12.1]{BGduality} 
 (noting that the sign on $d$ is incorrect in the statement of~\cite[12.1]{BGduality}). 
 More generally we can consider any noetherian (ungraded) 
 ring $R$ of finite global dimension and deduce that $C^*(BG;R)$ has local Gorenstein 
 duality. In this case we can only determine the shift function on maximal ideals. 
 For instance, we can take $R=\Z$ and deduce that $C^*(BG;\Z)$ has local 
 Gorenstein duality of shift $\m \mapsto g-1$.
\end{ex}

\begin{ex}[Complex $K$-theory]\label{ex:ku-KU}\leavevmode
\begin{itemize}
\item[(i)] Let $R = ku$ be the connective complex $K$-theory spectrum so that 
 $\pi_*(ku) = \Z[v]$ where $v$ is the Bott element in degree 2. 
 This is algebraically Gorenstein of shift $-2$ by direct inspection or by 
 applying Corollary~\ref{cor-poly-shift-gor}. 
 Therefore $C^*(BG;ku)$ has local Gorenstein duality with shift 
 function $\m \mapsto g -4$. 
 \item[(ii)]
 We can also consider the periodic $K$-theory spectrum $R = KU$ where 
 $\pi_*(KU) = \Z[v,v^{-1}]$. The ring $\Z[v,v^{-1}]$ is Gorenstein of shift $0$ 
 by Proposition~\ref{prop:periodic} and therefore $C^*(BG;KU)$ has local Gorenstein 
 duality with shift function $\m \mapsto g-1$.
 \end{itemize}
\end{ex}

\begin{ex}[Lubin-Tate spectra]
  Let $k$ be a perfect field of characteristic $p$, and write $W(k)$ for the Witt 
  vectors. The ring $W(k)$ is a regular local ring of Krull dimension 1 (as it is a 
  discrete valuation domain). 
  The Lubin-Tate spectrum $E_n$ is a commutative noetherian ring spectrum 
  with $\pi_*(E_n) = W(k)\llbracket u_1,\ldots,u_{n-1}\rrbracket [\beta^{\pm 1}]$ where $|u_i|=0$ and $|\beta|=2$, 
  so it is algebraically Gorenstein of shift $0$ by Proposition~\ref{prop:periodic}.  
  Therefore $C^*(BG; E_n)$ has local Gorenstein duality of shift $\m \mapsto g -n$.
\end{ex}

\begin{ex}[Topological Hochschild homology]\label{ex:THH}
 Let $R=\THH(\mbb{F}_p)$ be the topological Hochschild homology spectrum of $\mbb{F}_p$. 
 B\"{o}kstedt showed that $\pi_*\THH(\mbb{F}_p)=\mbb{F}_p[x]$ where $|x|=2$. 
 This is an algebraically Gorenstein ring
 of shift $-2$ so the ring spectrum $C^*(BG; \THH(\mbb{F}_p))$ has local Gorenstein 
 duality of shift $\m \mapsto g-3.$
\end{ex}

\begin{ex}[Topological modular forms]\label{ex:tmf}
We can also consider cases where $R$ is some variant of topological modular forms. 
We note that $\mrm{tmf}$ has a commutative model by the Goerss-Hopkins-Miller theorem, 
and the versions with level structures have commutative models~\cite[\S 6]{HillLawson}, 
see also~\cite{LN1, LN2} in the case of $\mrm{tmf}_1(3)$. In the following examples we are 
implicitly working $p$-locally. 
\begin{itemize}
\item[(i)] If $p \geq 5$, we may take $R = \mrm{tmf}$ which is a commutative ring spectrum with $\pi_*(\mrm{tmf}) = \Z_{(p)}[c_4, c_6]$ where $|c_i| = 2i$. 
Therefore $\pi_*R$ has finite global dimension and has Gorenstein shift $ -20$. 
Thus $C^*(BG; \mrm{tmf})$ has local Gorenstein duality of shift $\m \mapsto g -23$.
\item[(ii)] If $p = 3$, we take $R = \mrm{tmf}_1(2)$ which is a commutative ring spectrum with $\pi_*(\mrm{tmf}_1(2)) = \Z_{(3)}[c_2, c_4]$ where $|c_i| = 2i$. 
This has finite global dimension and has Gorenstein shift $-12$. Therefore $C^*(BG; \mrm{tmf}_1(2))$ has Gorenstein duality of shift $\m \mapsto g -15$. 
\item[(iii)] If $p = 2$, we may take $R = \mrm{tmf}_1(3)$. This is a commutative ring spectrum with $\pi_*(\mrm{tmf}_1(3)) = \Z_{(2)}[c_1, c_3]$ where $|c_i| = 2i$. 
As before, this has finite global dimension and has Gorenstein shift $-8$. Thus $C^*(BG; \mrm{tmf}_1(3))$ has Gorenstein duality of shift $\m \mapsto g -11$.
\end{itemize}
\end{ex}

\begin{ex}[Topological modular forms again]
 Let $\mrm{TMF}(n)$ denote the periodic version of $\mrm{TMF}$ for elliptic curves over 
 $\Z[1/n]$-algebras with a full level $n$ structure. 
  \begin{itemize}
  \item[(i)] For $n=2$, the ring spectrum $\mrm{TMF}(2)$ is $4$-periodic and has
  noetherian coefficients given by
 \[
 \pi_*\mrm{TMF}(2) =\Z[1/2][s^{\pm 1}, (s-1)^{-1}][\alpha^{\pm 1}]
 \]
 where $|\alpha|=4$ and $|s|=0$, see~\cite[8.1]{MS2016}. 
 This is graded Gorenstein of shift $0$ 
 by Proposition~\ref{prop:periodic} and it has Krull dimension $2$ as it defines an 
 open in $\mrm{Spec}(\Z[1/2][s])$. Therefore $C^*(BG; \mrm{TMF}(2))$ has local 
 Gorenstein duality of shift $g-2$. 
 \item[(ii)] For $n=3$, the ring spectrum $\mrm{TMF}(3)$ is $2$-periodic with
 noetherian coefficients
 \[
 \pi_*\mrm{TMF}(3)=\Z[1/3,\zeta][t^{\pm 1}, (1-\zeta t)^{-1},(1+\zeta^2 t)^{-1}][\beta^{\pm 1}]
 \]
 where $\zeta$ is a primitive third root of unity and $|\beta|=2$, 
 see~\cite[8.2]{MS2016}. This is graded Gorenstein of shift $0$ by 
 Proposition~\ref{prop:periodic}, and it has dimension $2$ since it defines an open 
 in $\mrm{Spec}(\Z[1/3][\zeta,t])$. Therefore $C^*(BG; \mrm{TMF}(3))$ has 
 local Gorenstein duality of shift $g -2$. 
 \end{itemize}
\end{ex} 

\section{Descent for local Gorenstein duality}\label{sec:descent}
In this section we prove a descent result for local Gorenstein duality along relatively Gorenstein maps. We then illustrate this with some examples. We start by recalling 
the following key definition from~\cite{Balmerdescent, MathewGalois}.

\begin{defn}\label{def-descendable}
 We say that a map of commutative ring spectra $f\colon R \to S$ is \emph{descendable} 
 if the thick ideal of $\mod{R}$ generated by $S$ contains $R$; in symbols, 
 $R \in \mrm{Thick}_R^\otimes(S)$. 
 Note that we permit tensoring by arbitrary $R$-modules, rather than just the 
 compact ones. This is equivalent to the thick ideal of $\mod{R}$ generated by $S$ 
 being the whole of $\mod{R}$.
\end{defn}

In the case when the map $f\colon R \to S$ is finite, descendability reduces to a more familiar condition. Recall our use of the terminology build and finitely build from the 
conventions section. 

\begin{lem}\label{lem-finite+descendable}
Let $f\colon R \to S$ be a finite map of commutative ring spectra. Then $f$ is descendable if and only if $S$ builds $R$.
\end{lem}

\begin{proof}
If $S$ builds $R$, then $S$ must finitely build $R$ by Thomason's localization theorem~\cite[2.1.3]{Neeman} since $R$ and $S$ are compact $R$-modules. Conversely, if $f$ is 
descendable, then $R \in \mrm{Thick}_R^\otimes(S) \subset \mrm{Loc}_R^\otimes(S) = \mrm{Loc}_R(S)$ so $S$ builds $R$. 
\end{proof}

We will see some examples of descendable maps later on in this section. We refer the interested reader to~\cite{Mathewexamples} for a wider range of examples.

\subsection{The descent theorem} 
In order to descend local Gorenstein duality along a finite map $f\colon R \to S$ of commutative noetherian ring spectra we require the shift function of $S$ to be of a certain form as we now describe. Given $\p \in \Spech(\pi_*R)$, we let $\pb$ denote a prime ideal of $\pi_*S$ such that $\res_f(\pb) = \p$. Since $f$ is finite the lying over theorem ensures that such a prime does exist, but in general there will be many choices of such a prime. We therefore require the shift function of $S$ to be locally constant in the sense that each such choice of $\pb$ has the same shift. 
\begin{defn}\label{def-locally-constant}
Let $f\colon R \to S$ be a finite map of commutative noetherian ring spectra in which $S$ has local Gorenstein duality with shift function $\nu$. We say that the shift function $\nu$ of $S$ is \emph{locally constant relative to $f$} if for all $\p \in \Spech(\pi_*R)$ and $\q, \q' \in U = \mrm{res}_f^{-1}(\p)$ we have $\nu(\q) = \nu(\q')$.
\end{defn}

We now give our descent result. 
\begin{thm}\label{thm:descent}
Let $f\colon R \to S$ be a map of commutative noetherian ring spectra such that: 
\begin{itemize}
\item[(a)] $f\colon R \to S$ is relatively Gorenstein of shift $a$;
\item[(b)] $f$ is finite and descendable;
\item[(c)] $S$ has local Gorenstein duality with shift function $\nu$ locally constant relative to $f$.
\end{itemize}
Then $R$ has local Gorenstein duality with shift function $\p \mapsto \nu(\pb) + a$.
\end{thm}

%

\begin{proof}
 Fix $\p \in \Spech(\pi_*R)$, and put $U=\res_f^{-1}(\p)$ as before. 
 Since $f$ is relatively Gorenstein,  the extension of scalars functor $f_*$ preserves both products 
 and coproducts by Lemma~\ref{lem:ambi}. 
 Thus we can apply~\cite[7.10]{BIK2} and get a decomposition 
 $f_*\Gamma_\p R_\p \simeq \bigoplus_{\q \in U}\Gamma_\q S_\q$.  
 Using~\cite[4.9, Proof of 4.27]{BHV}, we also see that 
 $f_! \I_\p \simeq \bigoplus_{\q\in U}\I_\q$. 
 Since $\nu$ is locally constant relative to $f$, we know that $\nu(\q) = \nu(\q')$ 
 for all $\q, \q' \in U$. 

 Combining the above observations, we have the series of equivalences
 \[
 f_! \Gamma_\p R_\p \simeq \Sigma^a f_*\Gamma_\p R_\p \simeq \Sigma^a\bigoplus_{\q \in U} \Gamma_\q S_\q \simeq \Sigma^a \bigoplus_{\q \in U}\Sigma^{\nu(\q)}\I_\q \simeq \Sigma^{\nu(\q)+a} f_!\I_\p 
 \]
 where the first equivalence follows from Lemma~\ref{lem:ambi}. Using this equivalence and adjunction, for any $S$-module $N$ we obtain an equivalence of $R$-modules
 \begin{equation}\label{eq:coext}\Hom_R(N, \Gamma_\p R_\p) \simeq \Hom_S(N, f_!\Gamma_\p R_\p) \simeq \Hom_S(N, \Sigma^{\nu(\q)+a}f_!\I_\p) \simeq \Hom_R(N, \Sigma^{\nu(\q)+a} \I_\p)\end{equation}
which is natural in $N$. 

By~\cite[Proposition 2.12]{Mathewexamples}, for any $M \in \mrm{Thick}_R^\otimes(S)$, we have a natural equivalence \[M \xrightarrow{\simeq} \mrm{lim}\left(M \otimes_R S \doublearrows M \otimes_R S
\otimes_R S \triplearrows \cdots \right)\] exhibiting $M$ as a \emph{universal} limit. As this limit is universal, it remains a limit after applying any exact functor. As $f$ is descendable, we in particular have this for $M = R$. 

We apply the exact functors $\Hom_R(-,\Gamma_\p R_\p)\colon \mod{R} \to \mod{R}^\mathrm{op}$ and $\Hom_R(-, \Sigma^{\nu(\q)+a}\I_\p)\colon \mod{R} \to \mod{R}^\mrm{op}$ to the presentation of $R$ as a universal limit, to obtain equivalences \[\Gamma_\p R_\p \simeq \mrm{colim}\left(\Hom_R(S, \Gamma_\p R_\p) \leftdoublearrows \Hom_R(S \otimes_R S, \Gamma_\p R_\p) \lefttriplearrows \cdots \right)\] and \[\Sigma^{\nu(\q)+a}\I_\p \simeq \mrm{colim}\left(\Hom_R(S, \Sigma^{\nu(\q)+a}\I_\p) \leftdoublearrows \Hom_R(S \otimes_R S, \Sigma^{\nu(\q)+a}\I_\p) \lefttriplearrows \cdots \right).\]
Now (\ref{eq:coext}) yields a natural equivalence between the two colimit systems, and hence between the colimits. Therefore we obtain an equivalence $\Gamma_\p R_\p \simeq \Sigma^{\nu(\q)+a}\I_\p$ as required.
\end{proof}

\begin{ex}[Real $K$-theory]\leavevmode
\begin{itemize}
\item[(i)] By Wood's theorem, the complexification map $f\colon KO \to KU$ is finite and descendable~\cite[3.17]{Mathewexamples}, and is relatively Gorenstein of shift $-2$~\cite[19.1]{Greenlees2018}. Since $KU$ is algebraically Gorenstein of shift 0 by Proposition~\ref{prop:periodic}, it follows from Lemma~\ref{lem:alggor} that $KU$ has local Gorenstein duality with shift function $\nu(\q)=-\mrm{ht}(\q)=\mrm{coht}(\q)-1$. This is locally constant relative to $f$ by Lemma~\ref{lem:coheight} and therefore $KO$ has local Gorenstein duality with shift function $\p \mapsto \mrm{coht}(\p)-3$. 
\item[(ii)] One can apply the descent theorem to the connective version $f\colon ko \to ku$ as well. The map $f$ is again finite and descendable by Wood's theorem so we only need to check that the shift function is locally constant relative to $f$. In this case $\pi_*ku=\Z[v]$ with $|v|=2$ which is algebraically Gorenstein with shift function given by 
\[
\nu(\q)=\begin{cases}
-2 & \mathrm{if}\;v\in \q \\
0 & \mathrm{if}\; v \not \in \q.
\end{cases}
\] 
Therefore $ku$ has local Gorenstein duality with shift function $\nu(\q)-\mrm{ht}(\q) =\nu(\q)+ \mrm{coht}(\q)-2$ by the dimension formula (Proposition~\ref{prop:dimformula}). Since coheights are preserved under 
finite maps by Lemma~\ref{lem:coheight}, we only need to check that $\nu(\q)$ is locally constant relative to $f$. Observe that there exists an element $\beta\in ko_8$ such that $f(\beta)=v^4$. Therefore 
if $\q,\q' \in \Spech(\pi_*ku)$ are such that $f^{-1}(\q')=f^{-1}(\q)$ then either  
$v^4\in \q \cap \q'$ or $v^4 \in \q^c \cap (\q')^c$. In the first case, by the prime condition we see that 
$v$ is in both $\q$ and $\q'$ and therefore $\nu(\q) = \nu(\q')=-2$. In the second case, we claim that $v$ is in neither $\q$ nor $\q'$. A priori, we might have $v$ lying in one of the primes but not the other, but then $\beta \in f^{-1}(\q) = f^{-1}(\q')$ and therefore $v^4$ and hence $v$ are in both $\q$ and $\q'$ which is a contradiction. This proves the claim, and therefore in the second case we have 
$\nu(\q)=\nu(\q')=0$. Thus $\nu(\q)$ is locally constant relative to $f$.  
It follows that $ko$ has local Gorenstein duality with shift function $\p \mapsto \nu(\pb)-4+\mrm{coht}(\p)$.
\end{itemize}
\end{ex} 

\subsection{Towards descent for cochain spectra}
In the section, we make some remarks about the application of the local Gorenstein descent theorem to cochain spectra. Given a map $f\colon R \to S$ of commutative ring spectra we obtain an induced map $\theta\colon C^*(BG;R) \to C^*(BG;S)$ which one may attempt to apply descent to. 

\begin{lem}\label{lem:cochaindescendable}
Let $f\colon R \to S$ be a finite, descendable map of commutative ring spectra. Then the induced map $\theta\colon C^*(BG;R) \to C^*(BG;S)$ is a finite, descendable map of commutative ring spectra.
\end{lem}
\begin{proof}
 Firstly, note that since $f$ is finite 
 we have a natural equivalence
 \[
 C^*(BG;R)\otimes_R S \simeq C^*(BG;S).
 \] 
 By assumption $R$ finitely builds $S$, and applying the exact functor $C^*(BG;R) \otimes_R -$ and the above equivalence shows that $C^*(BG;R)$ finitely builds $C^*(BG;S)$. Therefore $\theta$ is finite. 
 
 Since $S$ builds $R$, applying the exact coproduct-preserving functor $C^*(BG;R) \otimes_R -$ and the equivalence in the previous paragraph shows that 
 $C^*(BG;S)$ builds $C^*(BG;R)$.  
 Therefore $\theta$ is descendable by Lemma~\ref{lem-finite+descendable}.
 \end{proof}
 
In light of the previous lemma, in order to apply the descent theorem to $\theta\colon C^*(BG;R) \to C^*(BG;S)$, it remains to check that the shift function of $C^*(BG;S)$ is locally constant relative to $\theta$. Since we do not have a closed form for the shift function of $C^*(BG;S)$ on prime ideals, we currently do not know how to verify this condition.

\bibliographystyle{plain}
\bibliography{reference}

\end{document}